\documentclass[a4paper,12pt]{amsart}
\usepackage{amsmath,amsfonts,amssymb,amsthm}

\usepackage{pstricks}
\usepackage{subfig}
\usepackage{graphicx}
\usepackage{color}
\usepackage[all]{xy}
\usepackage{epic}
\usepackage{url}
\usepackage{setspace}
\usepackage{booktabs}
\usepackage{longtable}
\usepackage{latexsym}

\newtheorem{theorem}{Theorem}[section]
\newtheorem{lemma}[theorem]{Lemma}
\newtheorem{proposition}[theorem]{Proposition}
\newtheorem{corollary}[theorem]{Corollary}

\theoremstyle{definition}
\newtheorem{definition}[theorem]{Definition}
\newtheorem{example}[theorem]{Example}

\newtheorem{remark}[theorem]{Remark}

\newcommand{\canBeErased}[1]{#1}

\definecolor{colValyaOK}{rgb}{0.1,0.40,0.13}

\newcommand{\cD}{\mathcal{D}}
\newcommand{\cL}{\mathcal{L}}
\newcommand{\cO}{\mathcal{O}}
\newcommand{\cA}{\mathcal{A}}

\newcommand{\cV}{\mathcal{V}}
\newcommand{\cE}{\mathcal{E}}

\newcommand{\cB}{\mathcal{B}}
\newcommand{\cF}{\mathcal{F}}
\newcommand{\cZ}{\mathcal{Z}}

\newcommand{\Cstar}{\mathbb{C}^*}

\newcommand{\bC}{\mathbb{C}}

\newcommand{\bQ}{\mathbb{Q}}
\newcommand{\bZ}{\mathbb{Z}}
\newcommand{\bP}{\mathbb{P}}

\newcommand{\bV}{\mathbb{V}}

\DeclareMathOperator{\GL}{GL}
\DeclareMathOperator{\SL}{SL}
\DeclareMathOperator{\rank}{rk}
\DeclareMathOperator{\tail}{tail}
\DeclareMathOperator{\innt}{int}

\DeclareMathOperator{\Hom}{Hom}

\DeclareMathOperator{\cadiv}{CaDiv}

\DeclareMathOperator{\orb}{orb}

\DeclareMathOperator{\Aut}{Aut}
\DeclareMathOperator{\sym}{Sym}
\DeclareMathOperator{\spec}{Spec}
\DeclareMathOperator{\Spec}{Spec}

\DeclareMathOperator{\supp}{supp}
\DeclareMathOperator{\id}{id}

\newcommand{\faceof}{\preceq}

\DeclareMathOperator{\Bl}{Bl}
\newcommand{\kG}{\Gamma}

\DeclareMathOperator{\kdirlim}{\kd \lim_{\raisebox{0.0ex}{$\longrightarrow$}}}

\DeclareMathOperator{\divisor}{div}

\newcommand{\ovl}{\overline}
\newcommand{\wt}{\widetilde}
\newcommand{\wh}{\widehat}

\newcommand{\kU}{V}
\newcommand{\kUU}{\CV}
\newcommand{\Utor}{\kU^{\toroidal}}
\newcommand{\cUtor}{\kUU^{\toroidal}}

\newcommand{\kN}{{H'}}
\newcommand{\wondY}{\ko{Y}}
\newcommand{\cf}{\Sigma} 
\newcommand{\cfX}{\Sigma^X} 
\newcommand{\xSigma}{\fanX}
\newcommand{\cfXprime}{\Sigma^{X'}} 
\newcommand{\fanX}{\Sigma_X} 
\DeclareMathOperator{\toroidal}{tor}
\DeclareMathOperator{\pre}{pre}
\newcommand{\Xtor}{X^{\toroidal}} %
\newcommand{\XtorHat}{\wh{X}^{\toroidal}} %
\newcommand{\fanXtil}{\Sigma_{\Xtil}} 
\newcommand{\Xtil}{\widetilde{X}} %
\newcommand{\fanY}{\Sigma_Y} 

\newcommand{\Colors}{\CC}   
\newcommand{\ColorsN}{\CC'} 
\newcommand{\UColorsTor}{\Delta_X^{\toroidal}} 
\newcommand{\UColorsTil}{\til{\Delta}_{{X}}} 
\newcommand{\UColorsN}{\Delta_{Y}} 
\newcommand{\colH}{D}     
\newcommand{\colN}{D'}    
\newcommand{\leftT}{{T}}
\newcommand{\rightT}{{\TT}}
\newcommand{\valC}{\cV}	  
\newcommand{\valCN}{\cV'}   
\newcommand{\CFC}{\cF_C} 
\DeclareMathOperator{\toric}{\TV}
\DeclareMathOperator{\ttoric}{\relTV}
\DeclareMathOperator{\TV}{\mathbb{TV}}
\DeclareMathOperator{\relTV}{\wt{\TV}}	
\newcommand{\shift}{\overline{\rho}}	
\DeclareMathOperator{\tits}{\phi}

\newcommand{\PP}{\mathbb P}
\newcommand{\TT}{\mathbb T}
\newcommand{\A}{\mathbb A}

\newcommand{\Z}{\mathbb Z}
\newcommand{\C}{\mathbb C}
\newcommand{\Q}{\mathbb Q}
\newcommand{\CA}{{\mathcal A}}

\newcommand{\CC}{{\mathcal C}}
\newcommand{\CF}{{\mathcal F}}

\newcommand{\CO}{{\mathcal O}}

\newcommand{\CV}{{\mathcal V}}

\newcommand{\CX}{{\mathcal X}}

\newcommand{\entspr}{\mathop{\widehat{=}}}  

\definecolor{intOrange}{rgb}{1.0,.310,.0} 
\DeclareMathOperator{\Grass}{Grass} 

\newcommand{\kB}{B} 
\newcommand{\kBalt}{B^-} 
\newcommand{\chQ}{/\hspace{-0.3em}/^{\mbox{\tiny ch}}} 

\newcommand{\til}[1]{\widetilde{#1}}
\renewcommand{\iff}{\Leftrightarrow}

\newcommand{\gHom}{\mbox{\rm Hom}}

\newcommand{\Mor}{\operatorname{Mor}}

\newcommand{\pFan}{\mathcal{S}} 
\newcommand{\pfan}{\til{\mathcal{S}}} 
\newcommand{\pDiv}{\mathcal{D}} 
\newcommand{\pf}{\mathfrak{f}} 
\renewcommand{\div}{\operatorname{div}}

\newcommand{\disjcup}{\sqcup} 
\newcommand{\Gl}{\operatorname{GL}}
\newcommand{\Sl}{\operatorname{SL}}
\newcommand{\PGL}{\operatorname{PGL}}

\newcommand{\kst}{\,|\;}

\newcommand{\kST}{\,\Big|\;}

\newcommand{\kd}{\displaystyle}

\newcommand{\surj}{\rightarrow\hspace{-0.8em}\rightarrow}

\newcommand{\kss}{\scriptscriptstyle}
\newcommand{\kbb}{{\kss \bullet}}
\newcommand{\ko}{\overline}

\newcommand{\rato}{-\hspace{-0.3em}\to}
\newcommand{\Pol}{\operatorname{Pol}}
\newcommand{\dual}{^{\scriptscriptstyle\vee}}
\DeclareMathOperator{\Quot}{Quot}
\DeclareMathOperator{\loc}{loc}

\newcommand{\bT}{\til{T}} 
\newcommand{\bN}{\til{N}} 

\newcommand{\bSigma}{\til{\Sigma}}
\newcommand{\yT}{T_Y} 
\newcommand{\yN}{N_Y} 
\newcommand{\proots}{R^+} 
\newcommand{\broots}{\nabla} 
\DeclareMathOperator{\stab}{stab}

\def\a{\alpha}

\newgray{gray2}{0.9}

\newcommand{\valuationConeGlTwo}{%
\psset{xunit=0.6cm,yunit=.6cm}
\begin{pspicture}(-2,-2)(2,2)%
\pspolygon[linecolor=white,fillstyle=solid,fillcolor=gray2]
(-2,-2)(2,2)(2,-2)%
\psline[linewidth=.8pt]{<->}(0,-2)(0,2)
\psline[linewidth=.8pt]{<->}(2,0)(-2,0)
\qdisk(0,0){2pt}
\qdisk(1,1){1.5pt}
\qdisk(1,0){1.5pt}
\qdisk(-1,-1){1.5pt}
\qdisk(-1,1){2pt}
\psset{linecolor=white}
\qdisk(-1,1){1.2pt}
\psset{linecolor=black}
\end{pspicture}}

\newcommand{\BUaffineFour}{%
\psset{xunit=0.6cm,yunit=.6cm}
\begin{pspicture}(-2,-2)(2,2)%
\pspolygon[linecolor=white,fillstyle=solid,fillcolor=gray2]
(-2,-2)(2,2)(2,-2)%
\pspolygon[linecolor=white,fillstyle=solid,fillcolor=lightgray]
(0,0)(2,2)(2,0)%
\psline[linewidth=1.2pt](2,2)(0,0)(2,0)
\psline[linewidth=.8pt]{<->}(0,-2)(0,2)
\psline[linewidth=.8pt]{->}(0,0)(-2,0)
\qdisk(1,1){1.5pt}
\qdisk(1,0){1.5pt}
\qdisk(-1,-1){1.5pt}
\qdisk(-1,1){2pt}
\psset{linecolor=white}
\qdisk(-1,1){1.2pt}
\psset{linecolor=black}
\end{pspicture}}

\newcommand{\BUprojectiveFour}{%
\psset{xunit=0.6cm,yunit=.6cm}
\begin{pspicture}(-2,-2)(2,2)%
\pspolygon[linecolor=white,fillstyle=solid,fillcolor=lightgray]
(-2,-2)(2,2)(2,-2)%
\psline[linewidth=.8pt]{->}(0,0)(0,2)
\psline[linewidth=.8pt]{->}(0,0)(-2,0)
\psline[linewidth=1.2pt](2,2)(0,0)(2,0)
\psline[linewidth=1.2pt](0,0)(-2,-2)
\qdisk(1,1){1.5pt}
\qdisk(1,0){1.5pt}
\qdisk(-1,-1){1.5pt}
\qdisk(-1,1){2pt}
\psset{linecolor=white}
\qdisk(-1,1){1.2pt}
\psset{linecolor=black}
\end{pspicture}}

\newcommand{\affineFour}{%
\psset{xunit=0.6cm,yunit=.6cm}
\begin{pspicture}(-2,-2)(2,2)%
\pspolygon[linecolor=white,fillstyle=solid,fillcolor=gray2]
(-2,-2)(2,2)(2,-2)%
\pspolygon[linecolor=white,fillstyle=solid,fillcolor=lightgray]
(-2,2)(2,2)(2,0)(0,0)%
\psline[linewidth=.8pt]{->}(0,0)(0,-2)
\psline[linewidth=.8pt]{->}(0,0)(-2,0)
\psline[linewidth=1.5pt](-2,2)(0,0)(2,0)
\qdisk(1,1){1.5pt}
\qdisk(1,0){1.5pt}
\qdisk(-1,-1){1.5pt}
\psset{linecolor=red}
\qdisk(-1,1){2pt}
\psset{linecolor=white}
\qdisk(-1,1){1.2pt}
\psset{linecolor=black}
\psline[linewidth=.8pt,linestyle=dotted](0,0)(2,2)
\end{pspicture}}

\newcommand{\projectiveFour}{%
\psset{xunit=0.6cm,yunit=.6cm}
\begin{pspicture}(-2,-2)(2,2)%
\pspolygon[linecolor=white,fillstyle=solid,fillcolor=lightgray]
(0,0)(-2,2)(2,2)(2,-2)(-2,-2)%
\psline[linewidth=.8pt]{->}(0,0)(-2,0)
\psline[linewidth=1.5pt](-2,2)(0,0)(2,0)
\psline[linewidth=1.5pt](0,0)(-2,-2)
\qdisk(1,1){1.5pt}
\qdisk(1,0){1.5pt}
\qdisk(-1,-1){1.5pt}
\psset{linecolor=red}
\qdisk(-1,1){2pt}
\psset{linecolor=white}
\qdisk(-1,1){1.2pt}
\psset{linecolor=black}
\psline[linewidth=.8pt,linestyle=dotted](0,0)(2,2)
\end{pspicture}}

\newcommand{\GrassTwoFour}{%
\psset{xunit=0.6cm,yunit=.6cm}
\begin{pspicture}(-2,-2)(2,2)%
\pspolygon[linecolor=white,fillstyle=solid,fillcolor=lightgray]
(-2,-2)(-2,2)(2,2)(2,-2)%
\psline[linewidth=1.5pt](-2,2)(0,0)(2,0)
\psline[linewidth=1.5pt](0,0)(0,-2)
\qdisk(1,0){1.5pt}
\qdisk(0,-1){1.5pt}
\psset{linecolor=red}
\qdisk(-1,1){2pt}
\psset{linecolor=white}
\qdisk(-1,1){1.2pt}
\psset{linecolor=black}
\psline[linewidth=.8pt,linestyle=dotted](-2,-2)(2,2)
\end{pspicture}}

\begin{document}

\title[Merging divisorial with colored fans]
{Merging divisorial with colored fans}

\author[K.~Altmann]{Klaus Altmann}
\address{Institut f\"ur Mathematik und Informatik,
        Freie Universit\"at Berlin,
        Arnimallee 3,
        14195 Berlin, Germany}
\email{altmann@math.fu-berlin.de}

\author[V.~Kiritchenko]{Valentina Kiritchenko}
\address{National Research University
         Higher School of Economics,
         Vavilova 7, Moscow 112312 Russia}
\address{Institute for Information Transmission Problems RAS}
\email{vkiritch@hse.ru}
\thanks{The second author was supported by Dynasty foundation, AG Laboratory NRU HSE, RScF grant 14-21-00053 (Sections 3 and 7)
MESRF grants ag. 11.G34.31.0023, MK-983.2013.1 and by RFBR grants 12-01-31429-mol-a, 12-01-33101-mol-a-ved.
This study was carried out within ``The National Research
University Higher School of Economics'' Academic Fund Program in 2013-2014,
research grant No. 12-01-0194''.}

\author[L.~Petersen]{Lars Petersen}
\address{Institut f\"ur Mathematik und Informatik,
         Freie Universit\"at Berlin
	 Arnimallee 3,
	 14195 Berlin, Germany}
\email{petersen@math.fu-berlin.de}

\begin{abstract}
Given a spherical homogeneous space $G/H$ of minimal rank, we provide a simple
procedure to describe its embeddings as varieties with torus action
in terms of divisorial fans.
The torus in question is obtained as the identity component of the quotient group $N/H$, where $N$ is the normalizer of $H$ in $G$.
The resulting Chow quotient is equal to
(a blowup of) the simple toroidal compactification of $G/(H N^\circ)$.
In the horospherical case, for example, it is equal to a flag variety,
and the slices (coefficients) of the divisorial fan are merely shifts
of the colored fan along the colors.
\end{abstract}

\maketitle

\section{Introduction}
\label{sec:introSP}

\subsection{}
We are working over the base field $\C$.
Normal varieties $X$ coming with an effective action of an algebraic
torus $\rightT$ -- also called $\rightT$-varieties --
can be encoded by divisorial fans
$\pFan^X=\sum_{D\subseteq Y} \pFan^X_D\otimes D$ on algebraic
varieties $Y$ of
dimension equal to the complexity of the torus action. In this notation
$D\subseteq Y$ runs through all prime divisors on $Y$,
and $\pFan^X_D$ denotes a combinatorial object associated to $D$
(being non-trivial for finitely many summands only).
Let $N$ denote the lattice of one-parameter subgroups
of $\rightT$. Every $\pFan^X_D$ stands for a polyhedral
subdivision of $N_\Q$ together with a prescribed labeling of its cells
referring to the set of affine charts covering $X$.
\\[1ex]
The $\rightT$-variety $X$ in question is then given as
a contraction of a toric
fibration over $Y$, and
the data $D$ and $\pFan^X_D$ describe exactly where and how this
fibration degenerates, respectively.
Vice versa $X$ can be reconstructed explicitly from $\pFan^X$
in two steps.
First one glues certain relative spectra over $Y$;
the result of this procedure is called $\ttoric(\pFan^X)$.
Finally, one obtains $X$ as $\toric(\pFan^X)$, which denotes
a certain birational contraction of $\ttoric(\pFan^X)$.
See Section \ref{sec:pDiv} for further details.

\subsection{}
\label{intro:spherical}
Let $G$ be a connected reductive group and $H \subset G$
a spherical subgroup such that the spherical homogenous space $G/H$ is of {\em minimal rank}
(see Definition \ref{def-rank}).
The goal of the present paper is to describe
spherical embeddings
$X\supseteq G/H$ by a divisorial fan $\pFan$,
i.e.\ $X=\toric(\pFan)$, on a modification
$Y$ of the simple, toroidal, and hence often wonderful
compactification $\wondY\supseteq G/\kN$ with $\kN:=H\cdot N_G(H)^\circ$,
cf.\ (\ref{sub:ColFans}),(\ref{subsec:toroidal}).
The latter spaces are very well understood; for any
$G$ there are only finitely many of them, and the modification
$Y\to\wondY$ is given by a certain fan $\fanY$
refining the valuation cone
$\valC_{\kN} \cong \Q^\ell_{\geq 0}$ of $\wondY$
where $l$ denotes the rank of $G/\kN$.
The basic tool for this construction will be the
Tits fibration $\tits:G/H\to G/\kN$. Its central fiber is the torus
$\rightT:=\kN/H$. It acts on $X$ from the right which turns
the spherical variety $X$ into a  $\rightT$-variety with $\pFan=\pFan^X$.
See Section~\ref{sec:sphVar} for further details about spherical varieties.
\\[1ex]
Fixing a Borel
subgroup $\kB\subseteq G$ such that $\kB\cdot H$ is open and dense in $G$,
denote by $\C(G/H)^{(\kB)}{\surj}\CX(G/H)$
the sets of $\kB$-semiinvariant functions and their
character lattice within $\CX_\kB:=\gHom(\kB,\C^*)$, respectively.
The dual lattices are connected by an exact sequence
$$
0 \to N \to \CX^*(G/H)\stackrel{p}{\to}\CX^*(G/\kN)\to 0,
$$
cf.\ \cite[Th\'{e}or\`{e}me 4.3(ii)]{BrionFrench} and
Proposition \ref{prop-lattices}.
Let $\Colors(G/\kN)$ denote the set of colors of $G/\kN$,
i.e.\ the set of $\kB$-invariant prime divisors of $G/\kN$.
After fixing a splitting of the above exact sequence,
our main result is the following

\begin{theorem}
\label{thm:main}
Let $X\supseteq G/H$ be a spherical embedding of minimal rank given by a colored fan $\cfX$
inside $\CX_\Q^*(G/H)$.
Denote by $\valC_H\subseteq\CX_\Q^*(G/H)$ the valuation
cone.
Then $X=\toric(\pFan)$ where $\pFan$ is a
divisorial fan on $(Y,N)$ with:
\\[0.5ex]
{\rm 1)}
The base space $Y$ is the toroidal spherical embedding of $G/\kN$
given by the (un-) colored fan $(\fanY,\emptyset)$
arising as {the image fan} {\rm (see Definition \ref{def-newFans})} of
$\cfX\cap\valC_H$ via the map $p$.
Its rays $a\in\fanY(1)$
correspond to the $G$-invariant divisors $D_a\subseteq Y$.
\\[0.5ex]
{\rm 2)} The maximal cells of the
divisorial fan $\pFan=\pFan^X$ describing $X$ as a $\rightT$-variety
are labeled by the maximal colored cones
$C=(C,\CFC)\in\cfX$ and the elements $w\in W$ of the Weyl group of $G$.
The part of $\pFan^X$ with label $(C,w)$ is equal to
$$
\pFan^X(C,w)=\sum_{a\in\fanY\!(1)}\hspace{-0.5em}\pFan^X_a(C)\otimes
D_a +
\hspace{-0.5em} \sum_{\colN \in \Colors(G/\kN)} \hspace{-0.8em}
(\shift(\colN)+\pFan^X_0(C))\otimes \ovl{\colN}
\;+\;
\hspace{-1.3em} \sum_{\colN \in \Colors(G/\kN)\setminus \CFC} \hspace{-1.5em}
\emptyset\otimes w\ovl{\colN}
$$
where $\pFan^X_a(C):=C\cap p^{-1}(a)$ is considered as an element
of $p^{-1}(a)\cong N_\Q$.
\end{theorem}

Note that the cells of the special fiber form
a fan $\pFan^X_0$.
Since it exhibits the asymptotic behavior of all other fibers
$\pFan^X_a$, we will sometimes also call it the tail fan $\tail(\pFan^X)$.
Let us furthermore point out that the coefficients of the colors
$\colN$ are just shifts of $\tail(\pFan^X)$. The shift vectors
$\shift(\colN)\in N$ are defined as projections to $N$
of the valuations $\rho_\colH\in \CX^*(G/H)$ corresponding to the colors of $G/H'$,
cf.\ (\ref{sub:toroidalThm}).
Finally, we would like to remark that the labeling
of the maximal cells is not quite
bijective. In (\ref{subsec:prelim1}) we will see that for a given
$C=(C,\CFC)\in\cfX$ the accompanying
$w\in W$ are rather parameterized by
$W/W_C$ for some subgroup $W_C\subseteq W$ depending on $\CFC$.
\\[0.5ex]

It would be interesting to generalize Theorem \ref{thm:main} to other spherical varieties.
As Example \ref{ex.GL_2/H} shows, the divisorial fan $\pFan^X$ should contain
additional maximal cells apart from those listed in Theorem \ref{thm:main}.
\subsection{}
\label{intro:applications}
We believe that merging these two partially combinatorial descriptions
via divisorial and colored fans
may help to obtain further results and insights into the realm of
spherical varieties, in particular concerning their deformation theory
(see e.g.\  \cite{toricDegSpherical}), and the computation of their Cox
rings.
As an example for the latter subject we would like to mention
that the Cox ring of a horospherical variety is
known to be a polynomial ring over the Cox
ring of a flag variety (\cite[Theorem 4.3.2]{coxWonderful}
or \cite[Theorem 3.8]{coxSpherical}).
However, an alternative way to understand
this might be to combine our
Theorem \ref{thm:main} with Theorem 1.2 from \cite{tvarcox}.

\subsection{}
\label{intro:organize}
The present paper is organized as follows. In Sections \ref{sec:pDiv} and
\ref{sec:sphVar}, we shortly review polyhedral divisors and
spherical varieties, respectively.
Section \ref{sec:sphVT} then
introduces the $\rightT$-action on a spherical
variety which is relevant for our purposes and was already announced
at the beginning of this section. Moreover,
the toroidal part of our main Theorem \ref{thm:main} appears there as
Theorem \ref{thm-toroidal}.
\\[0.5ex]
Sections \ref{sec:locTdown} and \ref{sec:locTriv}
contain the proof of Theorem \ref{thm-toroidal}. The main idea is
to reinterpret well-known facts from the spherical context within the context
of divisorial fans. Using the language of p-divisors allows us to recover
the encoded spherical variety directly from the given combinatorial data.
\\[0.5ex]
Section \ref{sec:invCol} finally deals with the non-toroidal case,
and we conclude by presenting several examples in
Section \ref{sec:examples}.

\subsection{Acknowledgment}
We are grateful to the referee for valuable comments and suggestions and for correcting an error in the first version of this paper.

\section{P-divisors and divisorial fans}
\label{sec:pDiv}


The upshot of \cite{tvar1, tvar2} is that normal varieties $X$ with a
complexity-$k$ action of an algebraic torus correspond to p-divisors
$\pDiv^X$ (for affine $X$) or divisorial fans $\pFan^X$ (for general $X$)
on a $k$-dimensional variety $Y$. The latter variety is the
so--called Chow quotient $Y=X\chQ \rightT$ and defined as a
GIT-limit quotient of the $\rightT$-action on $X$.
But any modification of $X\chQ \rightT$ could be taken as well.
Both data $\pDiv^X$ and $\pFan^X$ induce a diagram like
$$
\xymatrix@!0@C=5em@R=7.5ex{
Y & \til{X} \ar[l]_-{\pi} \ar[r]^-{r} & X
}
$$
where $r$ is a $\rightT$- equivariant proper birational contraction
resolving the indeterminacies of the rational quotient map $\pi:X\rato Y$.
While $X$ is obtained as $\toric(\pFan^X)$, the auxiliary
$\rightT$-variety $\til{X}$ shows up as $\ttoric(\pFan^X)$.
We are now going to recall this language in more detail.

\subsection{Polyhedral divisors}
\label{subsec:unConSimple}
Let $N\cong \Z^n$ be a free abelian group of rank $n$, and denote
its dual by $M:=\gHom(N,\Z)$. These data give rise to the torus
$\rightT=N\otimes_{\Z}\C^*$, and one can recover $M$ and $N$ as
its lattice of characters and 1-parameter subgroups, respectively.
Let us furthermore consider convex polyhedra
$\Delta\subseteq N_\Q := N\otimes_\Z\Q \cong \Q^n$.
The \emph{tail cone} of a polyhedron $\Delta$ is defined as
$$
\tail(\Delta):=\{a\in N_\Q\kst a+\Delta\subseteq \Delta\}.
$$
Note that the set of polyhedra with fixed tail cone $\sigma$ forms a
semigroup $\Pol^+(N,\sigma)$  with cancellation property (the addition
is given by the Minkowski sum).

\begin{definition}
\label{def-pDiv}
Let $Y$ be a normal, semi-projective (i.e.\
projective over an affine) variety and
fix a polyhedral, pointed cone $\sigma\subseteq N_\Q$.
A finite, formal sum
$\pDiv=\sum_D\Delta_D\otimes D$ is called a \emph{polyhedral divisor}
on $(Y,N)$ with $\tail(\pDiv)=\sigma$ if
\begin{enumerate}
\item
all $D$ are prime divisors on $Y$,
\item
all $\Delta_D\subseteq N_\Q$ are convex polyhedra with
$\tail(\Delta_D)=\sigma$, and
\item
for every $u\in\sigma\dual\cap M$ the evaluation
$\pDiv(u):=\sum_D \min\langle\Delta_D,u\rangle\cdot D$
is an element of the group of rational Cartier divisors
$\cadiv_\Q(Y)$ on $Y$.
\end{enumerate}
\end{definition}

\begin{remark}
(i)
The tail cone $\sigma$ serves as the neutral element
in $\Pol^+(N,\sigma)$, hence summands of the form
$\sigma\otimes D$ may be added or suppressed without having any
impact on $\pDiv$.
\\[0.5ex]
(ii)
On the other hand, we will also
allow $\emptyset$ as a possible coefficient.
While we define $\emptyset+\Delta:=\emptyset$, the summand
$\emptyset\otimes D$ indicates that the remaining sum is to be
considered on $Y\setminus D$ instead of $Y$.
This allows us to always ask for projective $Y$ although $\pDiv$
is only defined on its \emph{locus}
$\loc(\pDiv):=Y\setminus\bigcup_{\Delta_D=\emptyset} D$.
\\[0.5ex]
(iii)
Condition (3) is automatically fulfilled for $\Q$-factorial, in particular,
for smooth base varieties $Y$.
\end{remark}

Concavity of the $\min$ function, i.e.\
$\min\langle\Delta_D,u\rangle +\min\langle\Delta_D,v\rangle
\leq \min\langle\Delta_D,u+v\rangle$, implies that the
$M$-graded sheaf
$$
\CA:=\oplus_{u\in\tail(\pDiv)\dual\cap M} \;\CO_Y(\pDiv(u))
$$
carries the structure of an $\CO_Y$-algebra which induces
the following scheme over $Y$ or, actually, over $\loc(\pDiv)$.

\begin{definition}
\label{def-TVpDiv}
1) Let $\pDiv$ be a polyhedral divisor on $(Y,N)$. Then we call
$$
\ttoric(\pDiv):= \Spec_Y\,\CA \;\stackrel{\pi}{\to}
\loc(\pDiv)\hookrightarrow Y
$$
the \emph{relative} $\rightT$-variety associated to $\pDiv$.
It is affine if and only if $\loc(\pDiv)$ is.
\\[0.5ex]
2) $\pDiv$ is called {\em positive}
(or short ``p-divisor'') if $\pDiv(u)$ is semiample and big
on $\loc(\pDiv)$
for every $u\in\sigma\dual\cap M$ or $u\in\innt\sigma\dual\cap M$,
respectively. If this is the case, then we define its associated
\emph{absolute} $\rightT$-variety
$
\toric(\pDiv):= \Spec\,\kG(\loc(\pDiv),\,\CA)
$.
\end{definition}

We would like to remark that
also on a possibly non-complete variety $Y$
a rational Cartier divisor $D$ is called semiample if it admits a
basepoint-free multiple. Then, if $Y$ is semi-projective,
the spaces of sections $\kG(Y,\CO_Y(mD))$ are finitely generated
$\kG(Y,\CO_Y)$-modules, hence define maps to projective spaces
over $\kG(Y,\CO_Y)$. Their images do not depend on the choice of generators,
and we call a divisor $D$ big,
if $|mD|$ induces a birational morphism for $m\gg 0$.
It follows then
from \cite[Theorem 3.1]{tvar1} that
$\ttoric(\pDiv)\to\toric(\pDiv)$ are normal
varieties with the function field $\Quot\C(Y)[M]$. Moreover, the
$M$-grading of $\CA$ translates into a $\rightT$-action on both
varieties, and $\pi$ is a good quotient.
Finally, all normal, affine $\rightT$-varieties arise this way.

\subsubsection{Toric picture}
Affine toric varieties $X$ are $\rightT$-varieties of complexity $0$,
i.e.\ $Y = \textnormal{pt}$. The notion of a polyhedral
divisor collapses to its tail cone, that is a polyhedral cone
$\sigma \in N_\Q$ with $X = \toric(\sigma) = \ttoric(\sigma)$.

\subsubsection{$\C^* \curvearrowright \C^2$} \label{sub_Ex3actions}
For example, let us consider three different types of $\C^*$-actions
on the affine plane $\spec \C[x,y]$.
The latter are specified by their weights on the variables $x$ and $y$, respectively.
It is easy to check directly that these actions correspond to
the following polyhedral divisors
$\pDiv^\bullet = \Delta_0^\bullet \otimes 0 +
\Delta_\infty^\bullet \otimes \infty$ on $\bP^1$ such that
$\C^* \curvearrowright \C^2 = \toric(\pDiv^\bullet)$:\\[-1ex]

\begin{center}
\begin{tabular}{lcr}

\begin{tabular}{c|c|c}
$\deg x$ & $\deg y$ & type of action \\
\hline
1 & 0 & parabolic \\
\hline
1 & 1 & elliptic \\
\hline
1 & -1 & hyperbolic
\end{tabular}

& &

\begin{tabular}{c|c|c|c|c}
& $\Delta^\bullet_0$ & $\Delta^\bullet_\infty$ & tail cone & locus \\
\hline
$\pDiv^p$ & $[0,\infty)$ & $\emptyset$ & $[0,\infty)$ & $\bP^1 \setminus \infty$ \\
\hline
$\pDiv^e$ & $[1,\infty)$ & $[0,\infty)$ & $[0,\infty)$ & $\bP^1$ \\
\hline
$\pDiv^h$ & $[0,1]$ & $\emptyset$ & $\{0\}$ & $\bP^1 \setminus \infty$ \\
\end{tabular}

\end{tabular}
\end{center}

\subsection{Equivariant morphisms and divisorial fans}
\label{subsec:divFan}
Let $\pDiv'$ and $\pDiv$ be p-divisors on $(Y',N)$ and
$(Y,N)$, respectively. By \cite[Section 8]{tvar1},
$\TT$-equivariant maps $\ttoric(\pDiv')\to \ttoric(\pDiv)$
and $\toric(\pDiv')\to \toric(\pDiv)$
can be provided by a dominant map $\psi:Y'\to Y$ and a plurifunction
$\pf\in N\otimes \C(Y')$ such that
$\pDiv'\subseteq\psi^*\pDiv+\div(\pf)$. Here, both operators
$\psi^*$ and $\div$ are supposed to be applied to the divisors on $Y$
occurring in $\pDiv$ or the elements of $\C(Y')$ from $\pf$, respectively.
The inclusion sign is to be understood separately
for each of the polyhedral coefficients on both sides.
Moreover, it was shown in \cite[Section 8]{tvar1} that all
$\TT$-equivariant maps $\toric(\pDiv')\to \toric(\pDiv)$ arise in this
way when one allows to replace $\pDiv'$ on $Y'$ by
$p^*\pDiv'$ with an appropriate proper, birational map
$p:Y''\to Y'$ implying $\toric(p^*\pDiv')=\toric(\pDiv')$.
\\[1ex]
There is a special case which will play an important role later on.
Consider two p-divisors $\pDiv'=\sum_D\Delta'_D\otimes D$ and
$\pDiv=\sum_D\Delta_D\otimes D$ on $(Y,N)$ which satisfy
$\Delta'_D\subseteq \Delta_D$ for each $D$. Then we have
a $\rightT$-equivariant, open embedding
$\ttoric(\pDiv')\hookrightarrow \ttoric(\pDiv)$ if and only if
the polyhedra $\Delta'_y:=\sum_{D\ni y}\Delta'_D$ are faces of the
corresponding $\Delta_y:=\sum_{D\ni y}\Delta_D$ for all $y\in Y$,
cf.\ \cite[Prop 3.4, Remark 3.5(ii)]{tvar2}.
Moreover, it was also shown in loc.\ cit.\ that the condition
of $\toric(\pDiv')\hookrightarrow\toric(\pDiv)$ being an open embedding
implies this condition. If $\toric(\pDiv')$ is
an open subset of $\toric(\pDiv)$, then we will call
$\pDiv'$ a \emph{face} of $\pDiv$.

\begin{definition}\cite[Def 5.2]{tvar2}
A finite collection $\pFan$ of p-divisors on $(Y,N)$ is called
a \emph{divisorial fan} if for all $\pDiv,\pDiv'\in\pFan$ their intersection
$\pDiv\cap\pDiv'$ (taken via the polyhedral coefficients) is again a
p-divisor, a face of both $\pDiv$ and $\pDiv'$, and belongs to $\pFan$.
\end{definition}

Gluing all affine pieces together, the divisorial fan $\pFan$ gives rise to
the global $\rightT$-variety
$$
\toric(\pFan):=\kdirlim_{\pDiv\in\pFan}\hspace{-0.5em}\toric(\pDiv).
$$
Moreover, since all
coefficients $\Delta_D^{\pDiv}$ of $\pDiv$ ($\pDiv\in\pFan$)
fit into a polyhedral subdivision $\pFan_D$ of $N_\Q$, we may write
the divisorial fan as $\pFan=\sum_{D}\pFan_D\otimes D$.
In particular, all tail cones $\tail(\Delta_\kbb^{\pDiv})$ form
a fan $\tail(\pFan)$. The latter encodes the asymptotic behavior of
the \emph{slices} $\pFan_D$. However, to store all contents
of $\pFan$, it is still necessary to keep in mind which
cell of $\pFan$ belongs to which p-divisor $\pDiv\in\pFan$. This
is what we previously referred to as the \emph{labeling}. Note
however, that only the maximal elements of $\pFan$ matter for
this kind of information.

\subsubsection{Toric picture}
Open embeddings in the toric world correspond to inclusions of faces on
the level of polyhedral cones. Since divisorial fans coincide with their
polyhedral tail fans in this particular setting, face relations of
polyhedral divisors turn out to be the usual face relations
for polyhedral cones.

\canBeErased{%
\subsubsection{$\C^* \curvearrowright \bV(\cO_{\bP^1}(n))$} \label{sub:blowup}
For example, let us consider the geometric line bundle $p:\bV(\cO_{\bP^1}(n))\to \bP^1$
associated to $\cO(n)$ over $\bP^1$.
We assume that $\Cstar$ acts with weight $1$ on the fibers of $\bV(\cO_{\bP^1}(n))$
and trivially on its zero section $\bP^1 \to \bV(\cO_{\bP^1}(n))$.
This action is given by the following two maximal polyhedral divisors:
\[
\cD^1 = [n,\infty) \otimes 0 + \emptyset \otimes \infty \textnormal{ and }
\cD^2 = \emptyset \otimes 0 + [0,\infty) \otimes \infty\,.
\]
They correspond to affine charts $p^{-1}(\bP^1\setminus \{\infty\})$ and $p^{-1}(\bP^1\setminus \{0\})$, respectively,
and are glued along the polyhedral divisor
$\cD^1\cap\cD^2=\emptyset \otimes 0+\emptyset \otimes \infty$ using the plurifunctions
$n\otimes z_1/z_0$ and $0\otimes 1$ on $\bP^1$.
}

\subsubsection{$\C^* \curvearrowright \PP^2$}
Let us consider $\bP^2$ as a $\C^*$-variety with the following
action on its homogeneous coordinates: $\deg z_0 = 1$, $\deg z_1 = 0$, and
$\deg z_2=2$.
Using the corresponding toric downgrade (see Subsection \ref{subsec:torDown}) yields a divisorial fan
$\pFan$ with the following three maximal elements
$\pDiv^i = \Delta^i_0 \otimes 0 + \Delta^i_\infty \otimes \infty$\,.\\[-1ex]

\begin{center}
\begin{tabular}{c|c|c|c|c}
& $\Delta^i_0$ & $\Delta^i_\infty$ & tail cone & locus \\
\hline
$\pDiv^1$ & $[-1,0]$ & $\emptyset$ & $\{0\}$ & $\bP^1 \setminus \infty$ \\
\hline
$\pDiv^2$ & $[0,\infty)$ & $[1/2,\infty)$ & $[0,\infty)$ & $\bP^1$ \\
\hline
$\pDiv^3$ & $(-\infty,-1]$ & $(-\infty,1/2]$ & $(-\infty,0]$ & $\bP^1$ \\
\end{tabular}
\end{center}

\subsection{Compatible group actions}
\label{subsec:compGrpAct}
Let $X$ be a $\rightT$-variety together with another group $G$ acting
on it. We say that $\rightT$ \emph{normalizes} the $G$-action if
$\rightT\subseteq N_{\Aut(X)}(G)$.
This means that $\rightT$ acts on both $X$ and $G$, and, moreover,
the $G$-action $m:G\times X\to X$ is $\rightT$-equivariant
(with respect to the diagonal action of $\rightT$ on the left hand side).
In particular, this morphism can be understood in terms of
(\ref{subsec:divFan}). If $X$ is given by a p-divisor $\pDiv$ on
some variety $Y=X\chQ \rightT$, then $G\times X$ is given by a p-divisor
on $(G\times X)\chQ \rightT$ which looks like the familiar
$G$-bundle $X\times^{\rightT}G$ over $Y$.
\\[1ex]
The actions of $G$ and $\rightT$ even commute if and only if the
$\rightT$-action on $G$ is trivial. If this is the case, then $G$
acts on $Y$, too, and the diagram
$$
\xymatrix@R=2.5ex@C=0.4em{
G \ar@{=}[d] & \times & X\hspace{0.3em} \ar@{.>}[d]
\ar[rrr]^-{m} &&& \hspace{0.3em} X \ar@{.>}[d]\\
G & \times & Y\hspace{0.3em}
\ar[rrr]^-{m} &&& \hspace{0.3em} Y
}
$$
commutes. In the language of (\ref{subsec:divFan}) this means that the
p-divisors $G\times \pDiv$ and $m^*\pDiv$ only differ by some polyhedral
principal divisor $\div(\pf)$. If $\pDiv=\sum_D\Delta_D\otimes D$,
then the two p-divisors equal
$\sum_D\Delta_D\otimes (G\times D)$ and
$\sum_D\Delta_D\otimes m^*D$, respectively. Since
$\div(\pf)$ can only shift the polyhedral coefficients by integral vectors,
this means that the $\Delta_D$ for non-$G$-invariant prime divisors
$D$ have to be almost trivial, i.e.\ shifted tail cones.
This occurs e.g.\ for the coefficients of the colors as pointed out
in Theorem \ref{thm:main}.

\section{Spherical varieties}
\label{sec:sphVar}

\subsection{}
\label{subsec:sphericIntro}
In this section, we provide background on spherical varieties and colored fans.
Spherical varieties are natural generalizations of toric varieties.
They appear when a torus action is replaced by an action of an arbitrary
connected reductive group $G$.
\\[1ex]
A normal variety $X$ with a $G$-action is called {\em spherical} if a Borel subgroup
$B\subset G$ has an open dense orbit in $X$.
Similarly to toric varieties, every spherical variety contains only finitely many $G$--orbits and even finitely many $B$--orbits \cite[Remark 2.2]{Knop}.
Well-known examples of spherical varieties include horospherical varieties  (e.g.\ toric and flag varieties) \cite{pasquier1} and symmetric varieties (e.g.\ complete collineations and complete quadrics) \cite{cp1,cp2}.
\\[1ex]
A spherical $G$-variety $X$ can be regarded as a partial
$G$-equivariant compactification of a spherical homogeneous
space $G/H$ (isomorphic to the open $G$-orbit of $X$).
In what follows, by an {\em embedding} of a spherical homogenous space $G/H$ we mean a spherical $G$-variety $X$ together with a point $x\in X$ such that the $G$-orbit of $x$ is open in $X$, and the isotropy subgroup of $x$ equals $H$.
By a {\em compactification} of $G/H$ we mean a complete embedding of $G/H$.
The classification of spherical varieties consists of two parts.
The first part amounts to classifying all $G$--equivariant embeddings of a given spherical homogeneous space $G/H$.
Similarly to toric varieties, embeddings of $G/H$ can be classified by fans together with an extra structure provided by colors \cite{LV}.
Below we shortly recall this classification following \cite{Knop}.
\\[1ex]
The second part amounts to the classification of all spherical
homogeneous spaces which was finished only recently
using D.~Luna's program.
An exposition of the main steps of this program can be found e.g. in \cite{Bravi}.
The classification of spherical homogeneous spaces is based on the classification of {\em wonderful varieties}.
Recall that a smooth complete $G$--variety with an open dense orbit is called {\em wonderful} (of {\em rank $r$}) if
\begin{enumerate}
\item the complement to the orbit is the union of $r$ smooth irreducible divisors $D_1$,\ldots, $D_r$
with normal crossings;
\item for any $I\subset\{1,\ldots,r\}$ the intersection $\bigcap_{i\in I} D_i$ is a non-empty $G$--orbit closure.
\end{enumerate}
In particular, there is a unique closed $G$--orbit $D_1\cap\ldots\cap D_r$.
Wonderful varieties are spherical (see \cite{Bravi} for references).

\subsection{Colored fans}
\label{sub:ColFans}
We now introduce definitions needed to formulate classification results.
Let $G/H$ be a spherical homogeneous space.
As in (\ref{intro:spherical}) we fix a Borel subgroup $\kB$ such that $1\in G/H$
belongs to the dense orbit, i.e.\ we assume that $\kB\cdot H$ is
open and dense in $G$.
A {\em color} is a $\kB$-invariant irreducible divisor in $G/H$.
Let $\Colors=\Colors(G/H)$ denote the set of colors of $G/H$.
The {\em weight lattice} $\CX:=\CX(G/H)$ of $G/H$ is the set of all characters
of $\CX_\kB=\gHom(\kB,\C^*)$ that occur as weights of
eigenvectors for the natural action of $\kB$ on the field of rational functions $\C(G/H)$.
The rank of the weight lattice $\CX$ is called the {\em rank} of $G/H$.
Since $G/H$ is spherical,
for each weight in $\CX$, there exists a unique (up
to scalars) $\kB$-semiinvariant rational function with this weight
(\cite[S.6]{Knop}).
These functions are regular on $B\cdot H$, hence
there is an exact sequence
$$
\xymatrix@R=0.1ex@C=1.7em{
1 \ar[r] &
\C^* \ar[r] &
\C(G/H)^{(\kB)} \ar[r] &
\CX(G/H) \ar[r] &
0.\\
&& f \mbox{ (with $f(1)=1$)} \ar@{|->}[r] &
\chi(f)=f|_\kB^{-1}
}
$$
Thus a valuation $v$ on $\C(G/H)$ with values in $\Z$ gives rise to a linear function $\rho_v$ on $\CX$.
In particular, colors $\colH$ give rise to elements
$\rho_{\colH}\in\CX^*:=\Hom(\CX,\Z)$.
Let $\valC$ denote the set of all $G$--invariant $\Z$-valued valuations.
It turns out that the map $\rho:\valC\to\CX^*$, $v\mapsto\rho_v$ is injective.
The convex hull of the image of $\valC_\Q$ in $\CX^*_\Q:=\CX^*\otimes_\Z \Q$
is called the {\em valuation cone}.
In what follows, we identify $\valC_\Q$ with its image.
\\[1ex]
By a result of Brion and Knop \cite[Theorem 6.4]{Knop},
there exists a root system in $\CX$ such that its simple roots $\alpha_1$, \ldots, $\alpha_r$ give linear equations on the
facets of $\valC$, that is,
$$
\valC=\{x\in\CX^*_\Q \, |\, x(\a_i)\le 0, i=1,\ldots,r\}.
$$
In particular, the valuation cone is always cosimplicial.

\begin{definition}
\label{def-colCone}
Let $\CF$ be a subset (possibly empty) of $\Colors$ such that
$\rho(\CF)$ does not contain $0$,
and let $C\subseteq\CX^*_\Q$ be a strictly convex, polyhedral cone.
The pair $(C,\CF)$ is called a
{\em colored cone} with the set of colors $\CF$ if
\begin{enumerate}
\item $C$  is generated by $\rho(\CF)$ and some elements of $\valC$, and if
\item the relative interior of $C$ intersects the valuation cone.
\end{enumerate}
\end{definition}

For instance, if $G/H$ is a torus, then $\valC=\CX^*_\Q$ and
$\Colors=\emptyset$. So every strictly convex polyhedral cone $C$ of
full dimension is a colored cone $(C,\emptyset)$.
The face relation among colored cones is defined as
$$
(C_1,\CF_1)< (C_2,\CF_2)
\;:\iff\;
C_1 \mbox{ is a face of }C_2 \mbox{ and }
\CF_1=\CF_2\cap\rho^{-1}(C_1).
$$
A finite, non-empty set $\cf$ of colored cones forms a {\em colored fan}
if first,
every face of $(C,\CF)\in \cf$ belongs to $\cf$,
and second, every $v\in\valC$
belongs to the interior of at most one cone $C$ with
$(C,\CF)\in \cf$.
This implies in particular that the intersection of two cones inside
$\valC$ is a common face of both.
\\[1ex]
Every spherical variety $X$ with an open dense orbit $G/H$ gives rise
to a colored fan $\cfX$.
Namely, $X$ can be covered by a finite number of simple  spherical varieties.
Recall that a spherical variety is {\em simple} if it contains a unique closed $G$-orbit.
Every simple spherical variety $X_0$ defines a colored cone $(C(X_0),\CF(X_0))$ as follows.
The set $\CF(X_0)$ is the set of all colors whose closure in $X_0$ contains the closed orbit.
The cone $C(X_0)$ is spanned by
$$
C(X_0) = \langle
\rho(\CF(X_0)),\;\rho(D_1), \ldots, \rho(D_r)\rangle
$$
where $D_1$, \ldots, $D_r$ are irreducible $G$-invariant divisors on $X_0$.
The colored fan $\cfX$ is then
the union of colored cones $(C(X_0),\CF(X_0))$ over all simple
$G$--invariant subvarieties $X_0\subset X$.
By results of \cite{LV}, the map $X\mapsto \cfX$ is a bijection
between isomorphism classes of spherical varieties
with an open dense orbit $G/H$ and colored fans in $\CX^*(G/H)_\Q$.
\\[1ex]
By definition of $\cfX$, there is a bijective correspondence between $G$--orbits in $X$ and colored cones in $\cfX$.
Closed orbits correspond to maximal colored cones.
Some further properties of $X$ can be read from the colored fan $\cfX$,
e.g. $X$ is complete if and only if the support $|\cfX|$ of the colored fan contains the valuation cone.
\\[1.0ex]
For two $G$--equivariant embeddings $X$ and $X'$ of $G/H$,
we say that $X$ {\em dominates} $X'$ if there exists a $G$--equivariant
morphism $X\to X'$.
This can also be read from the colored fans $\cfX$ and $\cfXprime$.
Namely, $X$ dominates $X'$ if the fan $\cfX$ fits into the
fan $\cfXprime$, that is, for every colored $(C,\CF)$ of $X$
there exists a colored cone $(C',\CF')$ of $\cfXprime$,
such that $C\subseteq C'$ and $\CF\subseteq\CF'$.
Note that we use the word ``dominate'' here
even for non-proper or non-surjective maps.

\subsection{Toroidal embeddings}
\label{subsec:toroidal}
There is a special class of spherical embeddings, namely, toroidal
embeddings, whose geometric properties are easier to study.

\begin{definition}
A $G$-equivariant embedding $X$ of $G/H$ is {\em toroidal} if it has no
colors, that is, the closure in $X$ of any color of $G/H$ does not contain
a closed $G$--orbit.
\end{definition}

In other words, all of the colored cones in the colored fan of $Y$ have
empty sets of colors. In particular, any toric variety is toroidal.
Wonderful varieties are toroidal.
Any embedding $X$ is dominated by a smallest toroidal one $\Xtor$
obtained by replacing every colored cone $(C,\CF)$ by the (un-)colored
cone $(C\cap\valC,\emptyset)$.
\\[1ex]
Smooth toroidal embeddings (also called {\em regular}) are the closest
relatives of smooth toric varieties.
If they are complete, then they
can also be covered by affine charts $\A^n$ (where $n=\dim G/H$)
so that the closures of codimension one $G$--orbits intersect each chart
by coordinate hyperplanes $D_1$,\dots, $D_r$
(where $r=\rank G/H$), and all intersections $\bigcap_{i\in I} D_i$ for
$I\subset\{1,\ldots,r\}$ are exactly the intersections of $\A^n$
with the closures of $G$--orbits.
These affine charts are translates of those defined in Proposition \ref{prop-covering}.
There is a more general notion of log-homogeneous varieties introduced in
\cite{BrionLogHom}.
From a geometric viewpoint, these are the nicest possible varieties among
all varieties with an almost homogeneous action of an algebraic group.
It turns out that if the group is linear,
then log-homogeneous varieties are exactly smooth toroidal varieties
(in particular, they are spherical), cf.\ \cite[Section 4]{Bremen}.
From a geometric point of view it is thus sometimes more natural to consider
toroidal embeddings than arbitrary spherical varieties.
\\[1ex]
If the valuation cone is strictly convex (hence, simplicial) then
there is a special compactification $\wondY\!_{\valC}$ of $G/H$ whose
colored fan is given by the valuation cone and all of its faces.
This compactification is called {\em standard}.
Note that the valuation cone is strictly convex if and only
if $N_G(H)/H$ is finite, cf.\ \cite[Theorem 7.1]{Knop} or
\cite[(4.4), Proposition 1]{BrionFrench}.
Those subgroups $H$ are called {\em sober}.
The standard compactification $\wondY\!_{\valC}$ is a unique both
simple and toroidal compactification of $X$, and hence, the only candidate for a
wonderful compactification of $G/H$.
To determine when $\wondY\!_{\valC}$ is wonderful is a difficult problem
which is not yet completely solved.
It is known that if $N_G(H)/H$ acts on the set of colors effectively
(e.g.\ $N_G(H)=H$) then $\wondY\!_{\valC}$
is wonderful \cite[Theorem 30.1]{timashev1}.
The converse is not true.
Note that $\wondY\!_{\valC}$ dominates any simple compactification
of $G/H$ and is dominated by any toroidal compactification of $G/H$.

\subsection{Horospherical varieties}
\label{subsec:horospherical}
We shortly discuss properties of horospherical varieties since they
will play a major role in Section \ref{sec:examples}.
For more details
on this subject the reader may consult \cite[Chapter 29]{timashev1} or
\cite{pasquier1}.

\begin{definition}
A closed subgroup $H \subset G$
is called \emph{horospherical} if it contains the unipotent radical
of some Borel subgroup
$\kBalt$.
In this case, $G/H$ is said to be a horospherical homogeneous
space. Analogously, we call a normal $G$-variety $X$ horospherical if it
contains an open $G$-orbit which is isomorphic to a horospherical
homogeneous space.
\end{definition}
In particular,
tori and complete rational homogeneous spaces are horospherical.

It follows from the Bruhat decomposition of $G$ that
horospherical varieties are also spherical, namely, the opposite Borel subgroup $B$ has an open orbit on $G/H$.
Moreover, for any horospherical subgroup $H \subset G$ there exists a unique
parabolic subgroup $P \supset \kBalt$ such that $H$ is the intersection of
the kernels of the characters of $P$.
Furthermore, we have $P = N_G(H)$.
In more detail, given $H \subset G$ and the maximal torus $T=B\cap\kBalt$,
there exists a subset $I$ of the
simple roots of $G$ such that $P$ is generated by $W_I$ and $\kBalt$, i.e.\
$P = P_I$. Here, $W_I$ denotes the subgroup of the Weyl group
$W = N_G(T)/T$ which is generated by the reflections associated to
the elements of $I$.
Even more, the lattice $\CX(G/H)$ can be identified
with the set of characters of $P$ whose restrictions to $H$
are trivial.\\[1ex]
It turns out that any horospherical homogeneous space $G/H$ is the total
space of a torus fibration over the flag variety $G/P$ where the fiber
$P/H$ equals the torus $\rightT$ with character lattice $\CX(G/H)$.
This fibration can be extended to the toroidal case, i.e.\ any toroidal
horospherical variety is of the form $G \times^P Y$ where
$Y\supseteq \rightT$ is a toric variety.
This feature may be regarded
as the main reason for why horospherical varieties are more amenable
to specific calculations than arbitrary spherical varieties.
Note also that $\CX=\CX_\kB$ and  $\valC = \CX(G/H)^*_\Q$ for a horospherical embedding
$G/H \subset X$ which ensures that its colored fan is an
honest polyhedral fan.

\begin{example} \label{ex.SL_2/U} The simplest example of a non-compact horospherical homogeneous space is $\SL_2/U$ with
\[
U = \left( \begin{array}{cc} 1 & 0 \\ * & 1 \end{array}\right) \subset
\left(\begin{array}{cc} * & 0 \\ * & * \end{array}\right) = \kBalt.
\]
Here, $P = \kBalt$, and $I = \emptyset$. The
homogeneous space $\SL_2/U$ is isomorphic to $\C^2 \setminus \{0\}$, and
$\Sl_2/P = \bP^1$ with the usual projection. Apart from the trivial embedding
$\bC^2 \setminus \{0\}$ of $\SL_2/U$ there are five nontrivial ones,
see Figure \ref{fig:CFembSL2/U} for their colored fans: \\[-1ex]

\begin{center}
\begin{tabular}{ccccccccc}
(a) $\Bl_{0}\bC^2$ & & (b) $\bC^2$ & & (c) $\bP^2 \setminus \{0\}$ & &
(d) $\Bl_{0}\bP^2$ & & (e) $\bP^2$
\end{tabular}
\end{center}

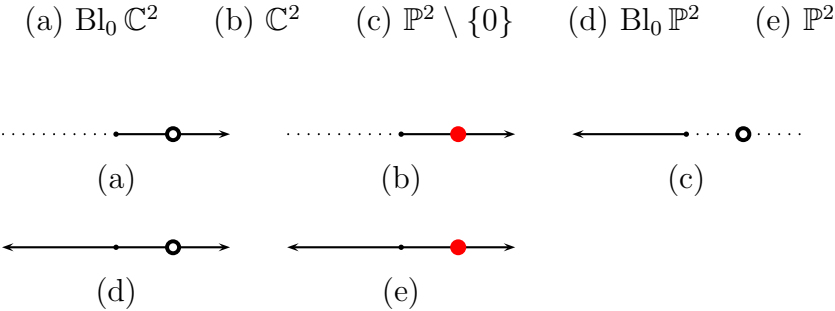
\begin{figure}[h]
\psset{unit=0.75cm}
\begin{pspicture}(0,-4.5)(15,0)

\psline[linestyle=dotted]{-}(0,-1)(2,-1)
\psline{->}(2,-1)(4,-1)
\qdisk(2,-1){1pt}
\qdisk(3,-1){3pt}
\psset{linecolor=white}
\qdisk(3,-1){1.5pt}
\psset{linecolor=black}
\uput[270](2,-1.3){(a)}

\psline[linestyle=dotted]{-}(5,-1)(7,-1)
\psline{->}(7,-1)(9,-1)
\qdisk(7,-1){1pt}
\psset{linecolor=red}
\qdisk(8,-1){3pt}
\psset{linecolor=black}
\uput[270](7,-1.3){(b)}

\psline[linestyle=dotted]{-}(12,-1)(14,-1)
\psline{<-}(10,-1)(12,-1)
\qdisk(12,-1){1pt}
\qdisk(13,-1){3pt}
\psset{linecolor=white}
\qdisk(13,-1){1.5pt}
\psset{linecolor=black}
\uput[270](12,-1.3){(c)}

\psline{<->}(0,-3)(4,-3)
\qdisk(2,-3){1pt}
\qdisk(3,-3){3pt}
\psset{linecolor=white}
\qdisk(3,-3){1.5pt}
\psset{linecolor=black}
\uput[270](2,-3.3){(d)}

\psline{<->}(5,-3)(9,-3)
\qdisk(7,-3){1pt}
\psset{linecolor=red}
\qdisk(8,-3){3pt}
\psset{linecolor=black}
\uput[270](7,-3.3){(e)}

\end{pspicture}
\caption{Colored fans associated to embeddings of $\SL_2/U$.}
\label{fig:CFembSL2/U}
\end{figure}

Note that the one-dimensional torus $\kN/H= H\cdot N_G^\circ(H)/H=P/U$
acts on these embeddings by scalar matrices.
Comparing (a)  with the first example in (\ref{sub:blowup}) for $n=1$ and and (b) with the second example in
(\ref{sub_Ex3actions}) we can check that Theorem \ref{thm:main} holds for the horospherical embeddings
$\Bl_{0}\bC^2$ and $\bC^2$, respectively.
\end{example}

\subsection{Spherical varieties of minimal rank}
\label{subsec:minimal_rank}
In what follows, we will mostly deal with spherical varieties of minimal rank.
This class of varieties include horospherical varieties and embeddings of $G$ (viewed as a homogeneous space under $G\times G$ acting by left and right multiplication).
In general, the rank ${\rm rk}(G/H)$ of a spherical homogeneous space $G/H$ satisfies the inequality
$${\rm rk}(G/H)\ge{\rm rk}(G)- {\rm rk}(H),$$
where ${\rm rk}(G)$ and ${\rm rk}(H)$ denote the ranks of the groups $G$ and $H$.
\begin{definition} \label{def-rank}
A spherical homogeneous space $G/H$ is {\em of minimal rank} if
$${\rm rk}(G/H)={\rm rk}(G)- {\rm rk}(H).$$
\end{definition}
Spherical homogeneous spaces of minimal rank are classified in \cite{Ressayre}.
They can be characterized by the following property:
\begin{proposition}\cite[Proposition 2.4]{Ressayre}\label{prop-Ressayre}
Let $T\subset G$ be a maximal torus.
A spherical homogenous space $G/H$ is of minimal rank if and only if for any toroidal embedding $X$ of $G/H$, the $T$-fixed points of $X$ lie in closed orbits.
\end{proposition}
This property is important for us because it yields a covering of $X$ by $T$-stable open affine subvarieties that can be explicitly described using the colored fan of $X$ and the Weyl group of $G$.
We now describe this covering.
First, let $X(C,\CFC)\subset X$ be the simple spherical embedding corresponding to a maximal colored cone $(C,\CFC)\in\cfX$.
Then the open $B$-invariant subvariety
$$
\wh{X}_{\id}(C,\CFC)=X(C,\CFC) \setminus
\bigcup_{D \in {\Colors(G/H)\setminus\CFC}}\hspace{-1.2em}\overline{D}
$$
is affine by \cite[Theorem~2.1]{Knop}.
For $w\in W$, put $\wh{X}_{w}(C,\CFC):=w\wh{X}_{\id}(C,\CFC)$.
This is a $T$-invariant subvariety (we assume that $T\subset B$).
\begin{proposition} \label{prop-covering}
If $X$ is a spherical variety of minimal rank with the colored
fan $\cfX$ then
$$X=\bigcup_{(C,\CFC)\in\cfX\hspace{-0.2em},\hspace{0.2em} w \in W}
\hspace{-1.2em}\wh{X}_w(C,\CFC),$$
where $(C,\CFC)$ runs through the maximal colored cones in $\cfX$.
This yields a covering of $X$ by $T$-invariant open affine subvarieties labeled
by the maximal colored cones and the elements of the Weyl group of $G$.
\end{proposition}
\begin{proof} W.l.o.g. assume that $X$ is complete.
The variety
$$X':=\bigcup_{(C,\CFC)\in\cfX, w \in W}\wh{X}_w(C,\CFC)$$
is open and $T$-invariant.
The intersection $X'\cap\cO$ with any closed $G$-orbit $\cO\subset X$ contains all $T$-fixed points in $\cO$.
Indeed, if $\cO$ corresponds to a maximal cone $C$ then $\cO\cap \wh{X}_{\id}(C,\CFC)$ is the open dense $B$-orbit in $\cO$ by \cite[Theorem 2.1 (c)]{Knop}.
Since $\cO$ is isomorphic to a flag variety, the open dense $B$-orbit in $\cO$ contains a unique $T$-fixed point $x_0$
and all other $T$-fixed points in $\cO$ have form $wx_0$ for $w\in W$.

Hence, if $X$ is toroidal, then $X'$ contains all $T$-fixed points by Proposition \ref{prop-Ressayre}.
It follows that the complement $X\setminus X'$ is empty.
Indeed, every nonempty closed $T$-invariant subvariety of $X$ must contain a $T$-fixed point by Borel's fixed point theorem.

The statement for a non-toroidal $X$ follows at once from the corresponding statement for a toroidal resolution of $X$.
\end{proof}

Below is an example of a complete spherical space (not of minimal rank) for which Proposition \ref{prop-covering} does not hold (this example was suggested to us by the referee).

\begin{example} \label{ex.P^1_P^1} Let $G=GL_2$.
Consider $X=\bP^1\times\bP^1$ as a $G$--variety under the diagonal action of $G$.
Then  Proposition \ref{prop-covering} yields only two charts out of four standard affine charts
for $\bP^1\times\bP^1$.
\end{example}
We will also need the following description of $B$-orbits in $G/H$.
For more details see \cite{Ressayre,BrionOrbits}.
Let $W_G$ and $W_H$ denote the Weyl groups of $G$ and $H$, respectively.
For arbitrary spherical homogeneous spaces, F.Knop defined an action of $W_G$ on $B$-orbits in $G/H$.

\begin{proposition}\cite[Propositions 2.1, 2.2]{Ressayre}\label{prop-Ressayre2}
The homogeneous space $G/H$ is of minimal rank if and
only if the $W_G$-action on its $B$-orbits is transitive.
Then there exists a unique closed $B$-orbit, and its stabilizer in $W_G$
is isomorphic to $W_H$.
\end{proposition}

This description generalizes the description of Schubert cells in flag varieties to all spherical varieties of minimal rank.
However, no such description is known for more general spherical varieties (cf. Example \ref{ex.P^1_P^1}).

As in the case of flag varieties, there are two competing ways to label $B$-orbits by elements of $W_G/W_H$: 
the identity element of $W_G$ labels either the closed (minimal) $B$--orbit or the open (maximal) $B$--orbit.
We will use the former labeling, which we now define more precisely.
Let $\cO_{\id}$ denote the closed $B$-orbit in $G/H$.
Note that this is not the $B$--orbit through the identity coset of $G/H$, since the latter is open by our choice of $B$.
In what follows, we denote by $\cO_{u}$ (where $u\in W_G$) the $B$-orbit obtained from $\cO_{\id}$ by the action of $u$.
The action of $W_G$ on $B$-orbits is related to the usual left action of $W_G$ on $G/H$ as follows.
If $\alpha$ is a simple root and $P_\alpha$ and $s_\alpha$ are the associated minimal parabolic subgroup and the simple reflection,
respectively, then two cases can occur:
\begin{enumerate}
\item $P_\alpha\cO_u=\cO_u$, then $\cO_u=\cO_{s_\alpha u}$ and $s_\alpha(\cO_u)=\cO_u$;

\item $P_\alpha\cO_u=\cO_u\sqcup\cO_{s_\alpha u}$ and the natural map $P_\alpha\times^{B}\cO_u\to P_\alpha\cO_u$ is birational.
If $\dim\cO_u<\dim\cO_{s_\alpha u}$ (that is, $\dim\cO_u+1=\dim\cO_{s_\alpha u}$) then
$s_\alpha(\cO_u)\subset \cO_{s_\alpha u}$.
\end{enumerate}

There is an analog of the weak Bruhat order on $B$-orbits in $G/H$, namely, $\cO_u\prec\cO_v$ if  $\dim\cO_{u}+1=\dim\cO_v$ and
$v=s_\alpha u$ for a simple root $\alpha$.
In particular, $\cO_u\subset\ovl{\cO_v}$.
In what follows, $\Gamma(G/H)$ denotes the oriented graph associated with the relation $\prec$, that is,
the vertices of $\Gamma(G/H)$ are $B$-orbits and two vertices $\cO_u$ and
$\cO_v$ are connected by an oriented edge labeled by $\alpha$ if
$\cO_u\prec\cO_v$ and $v=s_\alpha u$.
The graph $\Gamma(G/H)$ has a unique maximal element (the open dense $B$-orbit in $G/H$) and a unique minimal element (the closed orbit $\cO_{\id}$).
In particular, there exists a strictly increasing path from $\cO_{\id}$ to any other orbit $\cO_u$, and the number of edges in such a path is equal to $(\dim\cO_u-\dim\cO_{\id})$.
Similarly, there exists a strictly decreasing path from the maximal orbit to $\cO_u$ and the number of edges in such a path is equal to ${\rm codim}~\cO_u$.

\begin{remark}\label{rem-orbits} The colors of $G/H$ are closures of
codimension one $B$-orbits.
In the above notation, if
$\cO_{u_0}=(G/H)\setminus\bigcup_{\colH\in\Colors}\colH$ is the maximal
$B$-orbit in $G/H$
which has stabilizer $W_H\subseteq W_G$,
then the colors coincide with the closures of
$B$-orbits $\cO_{s_\alpha u_0}$ for simple reflections
$s_\alpha\in W_G\setminus W_H$.
Put $D(\alpha):=\ovl{\cO_{s_\alpha u_0}}$.
Then $s_\alpha(D(\alpha))\ne D$ for any color $D$ since $s_\alpha(\cO_{s_\alpha u_0})\subset\cO_{u_0}$.
On the other hand, if $D(\beta)\ne D(\alpha)$
is any other color, then $s_\alpha D(\beta)=D(\beta)$.
Indeed, $s_\alpha u_0\neq s_\beta u_0$ implies
$s_\alpha(s_\beta u_0)\neq u_0$ in $W_G/W_H$, i.e.\
$\dim \cO_{s_\alpha s_\beta u_0} < \dim \cO_{s_\beta u_0}$ or
$\cO_{s_\alpha s_\beta u_0}=\cO_{s_\beta u_0}$.
In either case, $\cO_{s_\alpha s_\beta u_0}\subseteq \ko{\cO_{s_\beta u_0}}=D(\beta)$.
Then,
$s_\alpha \cO_{s_\beta u_0}\subseteq
P_\alpha \cO_{s_\beta u_0} = \cO_{s_\beta u_0}\cup
\cO_{s_\alpha s_\beta u_0}\subseteq D(\beta)$.
The same argument shows that if $s_\alpha\in W_H$ then
$s_\alpha D(\beta)=D(\beta)$, too.
\end{remark}

We will also need the following result describing the orbits inside a given color.

\begin{lemma}\label{l.inclusion}
Let $\cO_u$ be a $B$-orbit in $G/H$ such that an increasing path $(\alpha_{i_1},\ldots, \alpha_{i_\ell})$
from $\cO_u$ to $\cO_{u_0}$ contains
an edge labeled by $\alpha$.
Then $\cO_u\subset D(\alpha)$.
\end{lemma}
\begin{proof}
Proceed by induction on $\ell={\rm codim}~\cO_u$.
If $\ell=1$, we have even  $\ovl{\cO_u}= \colH$.
Choose $j$ such that $\alpha_{i_j}=\alpha$.
By replacing $u$ with $s_{\alpha_{i_{j-1}}}\cdots s_{\alpha_{i_1}}u$ we may assume that $j=1$.
Put $\beta=\alpha_{i_{2}}$ and $v=s_\alpha s_\beta s_\alpha u$.
Then either $\cO_v\prec\cO_{s_\alpha v}=\cO_{s_\beta s_\alpha u}$
(and $\dim \cO_v=\dim\cO_u+1$) or
$\cO_{s_\beta s_\alpha u}\prec\cO_v$ (and $\dim\cO_v=\dim \cO_u+3$).
In the latter case,  $\cO_{s_\beta s_\alpha u}$
satisfies the induction hypothesis for $(l-2)$
and $\cO_u\subset\ovl{\cO_{s_\beta s_\alpha u}}$.
In the former case, $\cO_{v}$ satisfies the induction hypothesis
for $(l-1)$ so it remains to show that $\cO_u\subset\ovl{\cO_{v}}$.

The braid relation
$(s_{\alpha}s_{\beta})^{m(\alpha,\beta)}=\id$ within $W_G$
yields a closed path $\Pi$ in
$\Gamma(G/H)$ that goes through $\cO_u$ and $\cO_v$,
Note that any such path contains a unique {\em locally maximal} vertex $\cO_{w}$, that is,
$\cO_{s_\alpha w}\prec\cO_{w}$ and $\cO_{s_\beta w}\prec\cO_{w}$.
Indeed, $\ovl{\cO_w}$ is both $P_\alpha$- and $P_\beta$-invariant in this case, hence, $\ovl{\cO_w}$ contains
all the other orbits in the path.
This implies that $\Pi$ also has a unique minimal vertex $\cO_{{v_0}}$ with respect to $\prec$.
Indeed, let $\cO_{w_1}\in \Pi$ be a {\em locally minimal} vertex, i.e., $\cO_{s_\alpha w_1}\succ\cO_{w_1}$
and $\cO_{s_\beta w_1}\succ\cO_{w_1}$.
Let $\cO_{w_2}\in P$ be another locally minimal vertex.
Unless $\cO_{w_1}=\cO_{w_2}$, we can represent $\Pi$ as the union of two distinct paths $\Pi_1$ and $\Pi_2$ where
$\Pi_1$ goes from $\cO_{w_1}$ to $\cO_{w_2}$, $\Pi_2$ goes from $\cO_{w_2}$ to $\cO_{w_1}$, and $\Pi_1$ and $\Pi_2$ intersect
only at the endpoints.
For $i=1,2$, the first edge in $\Pi_i$ goes upwards and the last edge goes downwards, hence, $\Pi_i$ contains a locally maximal
vertex.
This contradicts to the uniqueness of a locally maximal vertex.

Let $\Pi_u$ and $\Pi_v$ be strictly increasing paths inside the loop from $\cO_{v_0}$
to $\cO_u$ and $\cO_v$, respectively.
The labels on both paths look like
$(\ldots,\alpha,\beta,\alpha,\beta,\ldots)$.
Denote by $P$ the parabolic group corresponding to the path
$\Pi_u$.
Moreover, we may assume that $\Pi_v=(s_\beta,\Pi_u)$.
Then, on the one hand, $P_\beta \cO_{{v_0}}=\cO_{{v_0}}\disjcup\cO_{s_\beta {v_0}}$.
On the other hand, $\cO_u$ is the open orbit in $P\cdot \cO_{{v_0}}$
and $\cO_v$ is the open orbit in $P\cdot \cO_{s_\beta {v_0}}$.
Thus, $\cO_u\subseteq P\cdot \cO_{{v_0}}\subseteq
P\cdot P_\beta\cO_{{v_0}}\subseteq P\cdot \ko{\cO_{s_\beta
{v_0}}}\subseteq\ko{\cO_v}$.
\end{proof}

\section{Towards the toroidal case}
\label{sec:sphVT}

\subsection{A new torus action}
\label{sub:newTorus}
We now introduce a torus action on embeddings of the spherical
homogeneous space $G/H$. Note that this won't be the restriction of the
$G$-action to a maximal torus $\leftT\subseteq G$. Instead, we use the
fact that $N_G(H)/H$ is a subgroup of a torus
\cite[Proposition 5.2]{Bremen}.
In particular, if we
put $\kN= H\cdot N_G^\circ(H)$, then $\rightT:= \kN/H$ is a torus, too.
Note that the group $\kN$ is the smallest sober subgroup that contains $H$ \cite[Lemma 30.1]{timashev1}.
The maximal linear subspace contained in the valuation cone $\valC$ has
dimension $\dim \rightT$ \cite[Theorem 7.1]{Knop}.
Note that $\rightT$ acts on $G/H$ from the right
and hence commutes with the left action of $G$.

\begin{lemma}
\label{lem-extendTaction}
The right action of $\rightT$ on $G/H$ extends to any $G$--equivariant
embedding of $G/H$.
\end{lemma}

\begin{proof}
The group $\wt G:=G\times \rightT$ acts on $G/H$ by left and right
multiplications as
$$
(g,t):x\mapsto gxt^{-1}
\hspace{1em}
(g\in G,\; t\in \rightT,\; x\in G/H).
$$
Hence, $G/H$ (so far only being considered as a homogeneous
$G$-space) may also be regarded as the homogeneous  $\wt G$-space
$\wt G/\wt H$,  where
$\wt H:=\{(th,t)\kst t\in \rightT, \,h\in H\}\cong H\times \rightT$.
Recall that we fixed a Borel subgroup $\kB\subset G$ such that
$\kB\cdot H$ is open and dense in $G$. It follows that
$N_G(H)$ is the stabilizer of $\kB H$ from the right, cf.\
\cite[Th\'{e}or\`{e}me 4.3(iii)]{BrionFrench}. In particular,
we obtain that $\kB H\supset \kN$ (indeed, $bH\cdot h'=bh'H=bb_1h_1H\in \kB H$).
Note that $\wt B:=\kB\times \rightT\subset \wt G$ is a
Borel subgroup in $\wt G$.
\\[0.5ex]
We now show that $G/H$ and $\wt G/\wt H$ have the same colors and
isomorphic weight lattices.
Indeed, the open $\kB$--orbit in $G/H$ is $\rightT$-invariant
since $\kB H\supseteq \kN$; hence,
any $\kB$--invariant irreducible divisor in $G/H$ is also
$\kB\times \rightT$--invariant.
Next, since the actions of $\kB\subseteq G$ and $\rightT$ commute,
each $\kB$--eigenvector $f\in\C(G/H)^{(\kB)}_\chi$ remains in this
one-dimensional space after applying the $\rightT$-action.
Hence, $f$ becomes a $\kB\times \rightT$--eigenvector of some weight
$\wt\chi\in\CX_{\wt B}$ lifting $\chi\in\CX_\kB$.
\\[0.5ex]
Since the resulting dual isomorphism
$\CX^*(G/H)\stackrel{\sim}{\rightarrow}\CX^*(\wt G/\wt H)$
is compatible with the identification
$\Colors(G/H)=\Colors(\wt G/\wt H)$ we stated before,
we obtain a bijection between the sets of colored fans
for $G/H$ and $\wt G/\wt H$, respectively.
Thus, every $G$--equivariant embedding of $G/H$ extends to
a $\wt G$--equivariant one of $\wt G/\wt H$.
\end{proof}

Since $\kN$ is sober, the homogeneous space $G/\kN$ admits the standard compactification $\wondY\!_{\valCN}$ (see Section \ref{subsec:toroidal}), which is wonderful in many cases of interest.
Moreover, blow-ups of the latter
will serve as base varieties for polyhedral divisors describing
spherical embeddings of $G/H$ as $\rightT$--varieties.

\subsection{Comparing $H$ and $\kN$}
\label{sub:comparison}
Let $M$ denote the character lattice of $\rightT$, and $N$ its dual,
i.e.\ the lattice of one-parameter subgroups of $\rightT$.

\begin{lemma}
\label{lem-multF}
Each $f\in \C(G/H)^{(\kB)}$ with $f(1)=1$ is partially multiplicative,
i.e.\ it satisfies $f(gh')=f(g)\cdot f(h')$ for $g\in G$ and $h'\in\kN$.
In particular, restriction to $\kN$ gives the vertical homomorphism
$\CX(G/H)\to M$
in the diagram
$$
\xymatrix@!0@C=8em@R=4.5ex{
& \CX(G/H)\; \ar@{^(->}[r] \ar@{->>}[dd]
& \CX_\kB\hspace{2em}
& \mbox{\rm (by restriction to $B$)}\\
\C(G/H)^{(\kB)}_{1\mapsto 1} \ar[ur]^-{\sim}
                           \ar[dr] \\
& M  \ar@{=}[r]
& \CX_{\rightT}\hspace{2em}
& \mbox{\rm (by restriction to $\kN$)}.
}
$$
\end{lemma}

\begin{proof}
From Lemma \ref{lem-extendTaction} it
follows that $f$ is an eigenvector
of $\kB\times\rightT$ and hence of $\kB\times\kN$.
Alternatively, one has the following direct argument:
For each $h'\in\kN$ we define a new rational function
$f'(g):=f(gh')$. It is not hard to see that $f'$ and $f$ transform
in the same ways with respect to the left $B$-action.
Moreover, since $\kN$ normalizes $H$, we see that $f'$ also is
$H$-invariant (for the action from the right). Hence,
$f$ and $f'$ differ multiplicatively by a constant,
i.e.\ $f'(g)=f(g)\cdot f'(1)$.
\end{proof}

The following result is derived from the \emph{Tits fibration}
$\tits:G/H\to G/\kN$.

\begin{proposition}
\label{prop-lattices}
The dual of the vertical homomorphism from {\rm (Lemma \ref{lem-multF})}
fits into the
short exact sequence
$$
0 \to  N  \to \CX^*(G/H) \stackrel{p}{\longrightarrow} \CX^*(G/\kN)\to 0.
$$
The valuation cone $\valC$ of $G/H$ is the full preimage of the
{\rm (}strictly convex{\rm )}
valuation cone $\valCN$ of $G/\kN$ under $p_\Q$.
Moreover, there is a natural identification of colors
$\tits:\Colors(G/H)\stackrel{\sim}{\longrightarrow}\Colors(G/\kN)$
compatible with $p=\tits_*$.
\end{proposition}

\begin{proof}
For the exactness of
$$
\;0 \to \CX(G/\kN) \to \CX(G/H) \to M \to 0
$$
see \cite[Th\'{e}or\`{e}me 4.3(ii)]{BrionFrench}.
Since $\kN\subseteq \kB H$, we know that $\tits^{-1}(\kB\kN/\kN)=\kB H/H$,
i.e.\ the two dense $\kB$-orbits correspond to each other via $\tits$.
The map $\tits$ is a locally trivial fibration with fiber
$\rightT$, cf.\ (\ref{subsec:TitsLocTriv}).
Hence, for every color $\colN\in\Colors(G/\kN)$ we know that
$\tits^{-1}(\colN)$ equals a single color
$\colH\in\Colors(G/H)$, and no colors from $G/H$ can
be sent via $\tits$ to a variety of larger codimension.
\\[0.5ex]
The equality of cones
$\valC=p^{-1}(\valCN)$ is also established
in \cite[4.3]{BrionFrench} by representing the valuation cone as the dual
of the cone generated by some negative roots,
cf.\ (\ref{sub:ColFans}), right before Definition \ref{def-colCone}.
Moreover, it is clear that for colors $\colH=\tits^{-1}(\colN)$ the associated
valuations $v_\colH$ and $v_\colN$ coincide on
$\C(G/\kN)$ understood as a subfield of $\C(G/H)$,
i.e.\ $p(\rho_{\colH})=\rho_{\colN}$ inside $\CX^*(G/\kN)$.
\end{proof}

Note that the previous exact sequences imply that
$G/H$ is of minimal rank if and only if $G/\kN$ has the same property.

\subsection{Introducing new fans}
\label{sub:comparisonX}
Let $X$, $X'$ be embeddings of $G/H$ and of $G/\kN$, respectively.
Generalizing a remark at the end of (\ref{sub:ColFans}) to the situation
of now two different subgroups $H$ and $\kN$,
we quote from \cite[Theorem 4.1]{Knop} that there exists a
$G$--equivariant map $X\to X'$ if and only if the fan  $\cfX$ maps to
the fan $\cfXprime$,
that is, for every colored $(C,\CF)$ of $X$
there exists a colored cone $(C',\CF')$ of $\cfXprime$,
such that $p(C)\subseteq C'$ and $p(\CF)\subseteq\CF'$.
For example, every toroidal embedding $X$ maps to the simple toroidal
compactification $\wondY=\wondY\!_{\valCN}$ already mentioned in
(\ref{sub:newTorus}).

\begin{definition}
\label{def-newFans}
Let $\cfX$ denote the colored fan which is associated with the
$G/H$-embedding $X$. Now we define the following ``ordinary" fans:
\begin{enumerate}
\item[(i)]
$\;\fanX:=\{C\cap\valC\kst (C,\CF)\in \cfX\}$
is called the underlying \emph{uncolored} fan.
\vspace{0.5ex}
\item[(ii)]
Let $\fanY=p(\fanX)$ denote the \emph{image fan} of $\fanX$ via $p$,
i.e.\ the
coarsest subdivision of the pointed cone $\valCN$ that refines all
images $p(C)$ of cones $C\in\fanX$.%
\vspace{0.5ex}
\item[(iii)]
Finally, let $\fanXtil$ be the coarsest common refinement of
$\fanX$ and $p^*\fanY:=
\{p^{-1}(C')\kst C'\in\fanY\}$.%
\end{enumerate}
\end{definition}

Let $\Xtor$ and $\Xtil$ denote
the toroidal $G/H$-embeddings corresponding to the
fans $\fanX$ and $\fanXtil$, respectively. Similarly, we call
$Y$ the toroidal $G/\kN$-embedding corresponding to the fan $\fanY$.
Invoking the remark on $G$-equivariant maps of spherical varieties at
the beginning of this section, we have the following $G$-equivariant
diagram connecting all these varieties:
$$
\xymatrix@!0@C=5em@R=7.5ex{
\Xtil \ar[r] \ar[d] & \Xtor \ar[r] \ar[d] & X\\
Y \ar[r] & \wondY.
}
$$
While the varieties and the maps of the first row carry the
natural $\rightT$-action provided by Lemma \ref{lem-extendTaction},
the torus $\rightT$ acts trivially on the second row. Recall that
the rays $a\in\fanY(1)$ correspond to the
$G$-invariant divisors $D_a$ of $Y$.

\subsection{Statement of the result}
\label{sub:toroidalThm}
In what follows, we assume that $G/H$ is of minimal rank.
Assume that $X$ is a spherical embedding of $G/H$ and consider
the varieties shown in the diagram of (\ref{sub:comparisonX}).
Fix a maximal torus $\leftT\subseteq \kB$.
Let $W:=N(\leftT)/\leftT$ denote the Weyl group.
The action of the Weyl group on $T$ by conjugations induces the action of $W$ on $\CX_{\leftT}$.
It is easy to check that $w\CX(G/H)=\CX(G/H^w)$ where $H^w:=wHw^{-1}$ and $B^w:=wBw^{-1}$.
As indicated in
$$
\xymatrix@R=1ex@C=1.7em{
0 \ar[r]
& \CX(G/\kN) \ar[r] \ar[dd]^-{w}
&  \CX(G/H) \ar[dr] \ar@{^(->}[rrr] \ar[dd]^-{w}
&&& \CX_{B} \ar@{=}[r] \ar[dd]^-{w}
& \CX_{\leftT} \ar[dd]^-{w}\\
&&& M  \ar[r] & 0 \\
0 \ar[r]
& w\CX(G/\kN) \ar[r]
&  w\CX(G/H) \ar[dr] \ar@{^(->}[rrr]
&&& \CX_{B^w} \ar@{=}[r]
& \CX_{\leftT}\\
&&& M^w  \ar[r] & 0,
}
$$
the $W$-action induces the isomorphism $\theta_w$ between $M$ and the character lattice $M^w$ of $\rightT^w:=\kN^w/H^w$.
We will use $W$ to build the divisorial fan $\pFan^X$
and $\pfan^X$ on $(Y,N)$.
Note that the affine covering defined in Proposition \ref{prop-covering} is $\rightT$-invariant.
Moreover, multiplication by $w\in W$ gives an isomorphism between the $\rightT$-varieties $\wh{X}_{\id}(C,\CFC)$ and  $\wh{X}_{w}(C,\CFC)$.
As the above diagram shows, the isomorphism does not affect the coefficients of the corresponding $p$-divisors in $N_\Q$.
Indeed, multiplication by $w$ gives the isomorphism between $X$ regarded as a $G/H$-embedding and $X$ regarded as a $G/H^w$-embedding, which takes the $\rightT$-variety $\wh{X}_{id}(C,\CFC)$  to the $\rightT^w$-variety $\wh{X}_{w}(C,\CFC)$.
Hence, if $\sum S_D\otimes D$ and $\sum S_{D'}\otimes D'$ are their respective p-divisors, then
$\sum S_{D'}\otimes D'=\sum \theta_w^*(S_{D})\otimes wD$.
If we now regard  $\wh{X}_{w}(C,\CFC)$ not as $\rightT^w$- but as $\rightT$-variety we have to apply the inverse of $\theta_w^*$ to the coefficients of its p-divisor thus getting $\sum S_{D}\otimes wD$.
\\[1ex]

We are going to define p-divisors corresponding to the affine charts in this covering.
To do so, fix a splitting of the exact sequence
$$
0 \to  N  \to \CX^*(G/H) \stackrel{p}{\rightarrow} \CX^*(G/\kN)\to 0
$$
from Proposition \ref{prop-lattices}
by choosing a cosection $s^*:\CX^*(G/H)\surj N$.
We use this projection to define $\shift(\colN)\in N$
as $s^*(\rho_\colH)$
where $\colH$ is the color of $G/H$ corresponding
to $\colN$ via the bijection
$\tits:\Colors(G/H)\stackrel{\sim}{\longrightarrow}\Colors(G/\kN)$
also established in Proposition \ref{prop-lattices}, i.e.\
satisfying $p(\rho_{\colH})=\rho_{\colN}$.
Moreover, we will use the splitting to always identify the
fibers $p^{-1}(a)$ with $p^{-1}(0)=N_\Q$.

\begin{definition}
\label{def-pFanX}
The maximal elements of $\pfan^X$ are p-divisors $\pfan^X(C,w)$
on $(Y,N)$ labeled by pairs of maximal colored cones $(C,\CF)\in\cfX$
(or, equivalently, by maximal ordinary cones $C\cap\valC\in\fanX$)
and elements $w\in W$.
They are defined as
$$
\pfan^X(C,w):=\sum_{a} \big(C\cap p^{-1}(a)\big)\otimes D_a
+
\sum_{\colN} \big(\shift(\colN)+(C\cap N_\Q)\big)\otimes \ovl{\colN}
+
\sum_{\colN} \emptyset\otimes w\ovl{\colN}
$$
where $a\in\fanY(1)$ runs through the (primitive generators of the) rays
of $\fanY$,
and
$\colN \in \Colors(G/\kN)$ runs through the colors of $G/\kN$.
Note that it makes the difference between $\pfan^X$ and
the $\pFan^X$ defined earlier
in Theorem \ref{thm:main} in (\ref{intro:spherical})
that now
$\colN $ runs through {\em all} the colors even in the third summand.
However, $\pfan^X=\pFan^X$  for toroidal $X$.
Note further that for $w=1$, the second summand will be annihilated by
the third one.
\end{definition}

The following theorem covers Theorem \ref{thm:main} for toroidal $X$,
i.e.\ for the case $X=\Xtor$.

\begin{theorem}
\label{thm-toroidal}
The divisorial fan $\pfan^X$ on $(Y,N)$ describes
$\Xtil$ and $\Xtor$ as $\rightT$-varieties, namely
$\Xtil=\ttoric(\pfan^X)=\ttoric(\pFan^X)$ and $\Xtor=\toric(\pfan^X)$.
\end{theorem}

The proof of this statement consists of a local part
(Section~\ref{sec:locTdown}), and a global one (Section~\ref{sec:locTriv}).

\begin{remark}
The $\ttoric$ construction is a local one, hence it yields the same
result for the arguments $\pfan$ and $\pFan$.
The description of the divisorial fan providing $\Xtil$
can be even more simplified, namely
$$
\Xtil = \ttoric\Big(\sum_{a} \big(\cfX\cap p^{-1}(a)\big)\otimes D_a
+ \sum_{\colN} \big(\shift(\colN)+\tail\big)\otimes \ovl{\colN}\Big).
$$
Note that no labels via cells of $\cfX$ or elements of $W$ are necessary.
This is due to the fact that Definition \ref{def-TVpDiv} of
$\ttoric(\pDiv)$ doesn't make any positivity assumptions on $\pDiv$.
However, the latter are necessary for the definition of $\toric(\pDiv)$.
Thus, we cannot expect simplifications for the descriptions of
$\Xtor$ or $X$.
\end{remark}

\section{Toric downgrades give the local picture}
\label{sec:locTdown}

\subsection{The toric skeleton}
\label{sub:skeleton}
The first step towards the proof of Theorem \ref{thm-toroidal}
is a local understanding of $\Xtor$ with respect to the right
$\rightT$-action. Denoting by $\UColorsTor$ the union of the closures
in $\Xtor$ of all colors of $G/H$,
the stabilizer of $\UColorsTor$
is given by a parabolic subgroup
$P := P(\UColorsTor)$
of $G$ which is actually independent of the particular toroidal embedding.
Furthermore it comes with a Levi decomposition $P = P_u\rtimes L$ such
that
\[
\Xtor \setminus \UColorsTor \cong P \times^L \TV(\xSigma) \cong P_u \times
\TV(\xSigma),
\]
where $\TV(\xSigma)$ denotes the ordinary toric variety associated to
the fan $\xSigma$, cf.\ \cite[Section 2.4]{BrionFrench} and
\cite[Theorem 29.1]{timashev1}.
The accompanying torus is equal to a quotient of $L$, and its character
lattice equals $\CX(G/H)$. Moreover, we may consider $\rightT = \kN/H$
as a subtorus which turns $\TV(\xSigma)$ into a $\rightT$-variety, cf.\
loc.cit.

The very same procedure also works for $\Xtil$ and $Y$. Moreover, it
is compatible with the morphisms shown in (\ref{sub:comparisonX}). Hence,
denoting the union of the closures of all colors of $G/H$ and $G/\kN$ in
the respective varieties by $\UColorsTil$ and $\UColorsN$, we obtain
the following commutative diagrams
$$
\xymatrix@!0@C=7.0em@R=4.5ex{
\Xtil \setminus \UColorsTil \ar[rd] \ar[ddd] \ar[rrr]^-{\hspace{3em}\sim}
&&& P_u\times\toric(\fanXtil) \ar[rd] \ar[ddd]   \\
& \Xtor \setminus \UColorsTor \ar[rrr]^-{\hspace{-4em}\sim}
&&& P_u\times\toric(\fanX)\\
{\ }\\
Y \setminus \UColorsN \ar[rrr]^-{\hspace{3em}\sim}
&&& P_u\times\toric(\fanY).
}
$$
Of particular interest to us is the right hand side of the diagram. There
we have two maps between three toric varieties multiplied with the
unipotent group $P_u$. In the next section, we will show how such a diagram between toric
varieties can be understood in the context of $\rightT$-varieties and
polyhedral divisors.

\subsection{Toric downgrades}
\label{subsec:torDown}
There is a very prominent procedure which gives rise to polyhedral
divisors and divisorial fans. This construction plays a fundamental role
in the proof of Theorem~\ref{thm-toroidal}, so we shortly recall it from
\cite[Section 8]{tvar0}.

Let $\rightT\subseteq\bT$ be a subtorus, and assume that we have fixed a
splitting of the corresponding exact sequence of 1-parameter subgroups
$$
0 \to N \to \bN \stackrel{p}{\to} \yN \to 0.
$$
Now, whenever $Z=\toric(\Sigma)$ is a toric variety given by a fan
$\Sigma$ in $\bN_\Q$,
then we define the fans $\fanY:=p(\Sigma)$ and
$\bSigma:=
\{C\cap p^{-1}(C')\kst C\in\Sigma,\, C'\in\fanY\}$ as in
Definition~\ref{def-newFans}(ii) and (iii), respectively.
They give rise to toric varieties $\til{Z}:=\toric(\bSigma)$ and
$\toric(\fanY)$. Similarly to the situation in (\ref{sub:skeleton}),
these varieties fit into the diagram
$$
\xymatrix@!0@C=6.0em@R=7ex{
\toric(\bSigma) \ar[r]^-{} \ar[d]^-{p} & \toric(\Sigma)\\
\toric(\fanY).
}
$$
The embedding $\rightT\hookrightarrow\bT$
turns $\til{Z}$ and $Z$ of the upper row into $\rightT$-varieties.
They can be described by a divisorial fan $\pFan$ on $(\toric(\fanY),\,N)$.
Let $\yT:=\yN\otimes_\Z\C^*$ denote the torus of the toric variety
$\toric(\fanY)$. It turns out that the divisors occurring in $\pFan$
are $\yT$-invariant, i.e.\ they are closures $\ko{\orb}(a)$
of $\yT$-orbits of codimension
one parameterized by the rays $a\in\fanY(1)$. We define
$$
\pFan:= \sum_{a\in\fanY(1)}\hspace{-0.3em}\pFan_a\otimes \ko{\orb}(a)
\hspace{1em}\mbox{with}\hspace{1em}
\pFan_a=\Sigma\cap p^{-1}(a),
$$
i.e.\ all $\pFan_a$ become polyhedral subdivisions of
$p^{-1}(a)=N_\Q$ with a naturally defined labeling.

\begin{proposition}\cite[Section 8]{tvar0}
\label{prop-toricDG}
The $\rightT$-structure of $\til{Z}\to Z$
is given by the divisorial fan $\pFan$
on $(\toric(\fanY),\,N)$, i.e.\
this morphism is equal to $\ttoric(\pFan)\to\toric(\pFan)$.
\end{proposition}

\subsection{The $\rightT$-variety $\Xtor\setminus\UColorsTor$}
\label{subsec:XwithoutColasTVar}
Combining results from (\ref{sub:skeleton}) and (\ref{subsec:torDown}),
we deduce that the $\rightT$-equivariant map
$(\Xtil \setminus \UColorsTil)\to(\Xtor \setminus \UColorsTor)$
is equal to $\ttoric(\pfan^X_1)\to\toric(\pfan^X_1)$
where $\pfan^X_1$ consists of the p-divisors
$\pfan^X\!(\kbb,\id_W)$ which were introduced in Definition~\ref{def-pFanX}:
\\[0.5ex]
The first summand is literally built by the recipe of the toric downgrade
of (\ref{subsec:torDown}); the former divisors
$\ko{\orb}(a)$ have just been replaced by
$P_u\times \ko{\orb}(a) = D_a\setminus \UColorsN$.
The second summand in $\pfan^X_1$ is void because of $w=1$, and the
presence of the last one just means that
the divisorial fan is supposed to be evaluated on
$Y \setminus \UColorsN$ instead of the entire complete $Y$.

\subsection{The action of the Weyl group}
\label{subsec:actWeyl}
Both spherical varieties $\Xtil\to\Xtor$ are covered by the open subsets
$(\Xtil \setminus w\UColorsTil)\to(\Xtor \setminus w\UColorsTor)$ where
$w\in W$ runs through all elements of the Weyl group. Since these charts
arise from
$(\Xtil \setminus \UColorsTil)\to(\Xtor \setminus \UColorsTor)$
by applying $w$, they are equal to
$\ttoric(\pfan^X_w)\to\toric(\pfan^X_w)$
with $\pfan^X_w:=w(\pfan^X_1)$, i.e.\
$$
\pfan^X_w:=\sum_{a\in\fanY} \big(C\cap p^{-1}(a)\big)\otimes D_a
+
\sum_{\colN\in\ColorsN} \emptyset\otimes w\ovl{\colN}.
$$
Gluing the charts of $\Xtil$ (and similarly of $\Xtor$) leads
to isomorphisms $\varphi_w$
$$
\xymatrix@!0@C=4.5em@R=4ex{
&& \Xtil \setminus \UColorsTil\hspace{1em}
&& \ttoric(\pfan^X_1+\emptyset\otimes w\UColorsN) \ar@{^(->}[rr]
                           \ar[dd]_-{\sim}^-{\varphi_w}
&& \ttoric(\pfan^X_1)\\
\Xtil \setminus (\UColorsTil \cap w\UColorsTil)
                           \ar@{^(->}[rru]\ar@{^(->}[rrd]\\
&& \Xtil \setminus w\UColorsTil\hspace{1em}
&& \ttoric(\pfan^X_w+\emptyset\otimes \UColorsN)\ar@{^(->}[rr]
&& \ttoric(\pfan^X_w).
}
$$
Note that we use $\emptyset\otimes\UColorsN$ as an abbreviation for
$\sum_{\colN\in\ColorsN} \emptyset\otimes \ovl{\colN}$, and recall from
(\ref{subsec:divFan}) what equivariant maps between $\rightT$-varieties
look like in terms of p-divisors or divisorial fans.
Since $\varphi_w$ induces the identity map $\id_Y$ on Y,
it corresponds to a plurifunction $\pf_w$ with
$$
\pfan^X_1+(\emptyset\otimes w\UColorsN)\subseteq
\pfan^X_w+(\emptyset\otimes \UColorsN) + \div(\pf_w).
$$
We cannot expect to have $\div(\pf_w)=0$. Otherwise all local isomorphisms
$\Xtil \setminus w\UColorsTil \cong wP_u \times \TV(\bSigma)$
would glue to a global one and thus expose $\TV(\bSigma)$ as a factor of
$\Xtil$. On the other hand, $\div(\pf_w)$ clearly has to vanish on those
slices where both divisorial fans already agreed in the first place. This
observation shows that
$\supp(\div\pf_w)\subseteq \UColorsN \cup w\UColorsN$.
Furthermore, we see that coefficients of the principal polyhedral divisors
$\div\pf_w$ are just shifts of the tail fan. Using this ``hint", we
correct the previous definition by
$$
\pfan^X_w:= \pfan^X_w + \div(\pf_w).
$$
Then we still have that
$\Xtil \setminus w\UColorsTil\cong \ttoric(\pfan^X_w)$. But
the gluing of $\Xtil \setminus \UColorsTil$ and
$\Xtil \setminus w\UColorsTil$ now simply corresponds to the inclusion
$
\pfan^X_1+(\emptyset\otimes w\UColorsN)\subseteq
\pfan^X_w+(\emptyset\otimes \UColorsN).
$
Thus, the corrected $\pfan^X_w$ fit into a huge common divisorial fan
$\pfan^X_{\pre}$ with $\pfan^X_{\pre}(\kbb,w)=\pfan^X_w$. Up to now
we have  proven that $\Xtil=\ttoric(\pfan^X_{\pre})$ and
$\Xtor=\toric(\pfan^X_{\pre})$. Yet, in contrast to the definition of
$\pfan^X$ in Definition \ref{def-pFanX} in (\ref{sub:toroidalThm})
we have that
$$
\pfan^X_{\pre}(C,w)=\sum_{a} \big(C\cap p^{-1}(a)\big)\otimes D_a
+
\sum_{\colN} \big(l_{\colN,w}+(C\cap N_\Q)\big)\otimes \ovl{\colN}
+
\sum_{\colN} \emptyset\otimes w\ovl{\colN}
$$
for certain elements $l_{\colN,w}\in N$. To complete the proof of
Theorem \ref{thm-toroidal}, it remains to check that these elements
do not depend on $w$ and are equal to $\shift(\colN)$.

\section{Concluding the toroidal case}
\label{sec:locTriv}

\subsection{}
\label{subsec:TitsLocTriv}
Restricting the map $\tits:\Xtil\to Y$ introduced in (\ref{sub:comparisonX})
to $\tits^{-1}(G/\kN)\to G/\kN$, we obtain a locally trivial fibration.
This
is well-known (following from $G$-homogeneity), and it is also
visible in the description
of $\Xtil$ as $\ttoric(\pfan^X_{\pre})$ in (\ref{subsec:actWeyl}) --
it is reflected by the fact that, after restricting
the divisorial fan $\pfan^X_{\pre}$ to $G/\kN$, all its remaining
polyhedral coefficients are shifted tail fans only.
\\[1ex]
The map $\tits^{-1}(G/\kN)\to G/\kN$ extends the classical
Tits fibration $\tits:G/H\to G/\kN$. In particular, both share the same twist,
which is encoded in the lattice elements $l_{\colN,w}\in N$
introduced at the end of (\ref{subsec:actWeyl}).
The only difference between the divisorial fans describing
$\tits^{-1}(G/\kN)$ and $G/H$ can be found in their tail fans which are
$\tail(\pFan)$ and $\{0\}$, respectively.
\\[1ex]
We exploit this relation to determine the shift vectors
$l_{\colN,w}\in N$ by presenting a polyhedral divisor $\pDiv^{G/H}$
supported on the colors $\Colors(G/\kN)$ on $G/\kN$ such that
$G/H \cong \relTV(\pDiv^{G/H})$ under the right action of the torus
$\rightT = \kN/H$. In particular, in this section we will forget about the
embeddings $\Xtil$, $\Xtor$, and $X$ discussed before --
we just focus on the original Tits fibration.

\subsection{The Tits fibration}
\label{sub:bundleEL}
By abuse of notation, let $\tits$ also denote the $\bQ$-linear extension
$\bQ^{\Colors(G/H)} \to \bQ^{\Colors(G/\kN)}$
of the natural identification of colors $\tits: \Colors(G/H)
\stackrel{\sim}{\longrightarrow} \Colors(G/\kN)$.
Recall further from Proposition \ref{prop-lattices}
and its proof in (\ref{sub:comparison})
that we have an exact sequence
$$
0 \to \CX(G/\kN) \to \CX(G/H) \to M \to 0
$$
together with a splitting associated to a
section $s: M \to \CX(G/H)$.
Moreover, given a character
$\chi \in M$, we fix an associated eigenfunction
$f_{s(\chi)} \in \bC(G/H)^{(\kB)}_{s(\chi)}$ on $G/H$.
In other words, it satisfies
$f_{s(\chi)}(b^{-1}gH) = s(\chi)(b)\cdot f_{s(\chi)}(gH)$.
We now define
\[
\begin{array}{c}
\cL(\chi) := \cO_{G/\kN}(\tits(\divisor f_{s(\chi)}))
= \cO_{G/\kN}\big(\sum_{\colN \in \Colors(G/\kN)}
\langle s(\chi),\rho_{\tits^{-1}(\colN)}\rangle \colN\big),
\end{array}
\]
where, as before,
$\rho_{\colH}=\rho_{\tits^{-1}(\colN)} \in \CX^*(G/H)_\bQ$
denotes the restriction of the
valuation associated to the color $\colH=\tits^{-1}(\colN) \in \Colors(G/H)$.
Note also that by $\tits(\sum_{D\in\Colors(G/H)}a_D D)$ we mean $\sum_{D'\in\Colors(G/H')}a_D D'$ by using our identification of colors $D=\tits^{-1}(D')$
in $G/H$ and $G/\kN$.
\\[1ex]
On the other hand, choosing a basis $\cB_M$ of $M$ we may embed
$\rightT = \Hom_{\textnormal{group}}(M,\bC^*)$ inside $\bC^m$ with
$m := \rank M$. Note that the action of $\rightT$ on itself extends to an
action on $\bC^m$ such that
$\bC^m = \oplus_{\chi \in \cB_M}\bC_{\chi}$ as a $\rightT$-module where $\rightT$
acts on $\bC_{\chi}$ via the character $\chi \in \cB_M \subset M$.
Hence we obtain the following embedding of $\rightT$-varieties
\[
G/H = G \times^{\kN} \rightT \subset
G \times^{\kN}\hspace*{-1ex} (\oplus_{\chi \in \cB_M}\bC_\chi) =: E \,.
\]
Let $\cE$ denote the sheaf of sections of $E$. Note that
it is equal to $\oplus_{\chi \in \cB_M} \cE(\chi)$ where $\cE(\chi)$
denotes the sheaf of sections of
$G \times^{\kN}\hspace*{-0.2ex} \bC_{\chi}$.

\begin{lemma}
\label{lem-LE}
$\cL(\chi) \cong \cE(-\chi)$, namely
$f_{s(\chi)}\cdot \cL(\chi) = \cE(-\chi)$
as subsheaves of $\C(G/H)$.
\end{lemma}

\begin{proof}
Given an open subset $U \subset G/\kN$ we have that
$$
\Gamma(U,\cE(\chi))  = \Mor_{\kN}(\pi^{-1}(U),\bC_\chi)
=
\big\{\eta \in \cO_G(\pi^{-1}(U)) \,|\, \eta \cdot h' = \chi(h')\eta \big\}
$$
where $\pi$ denotes the projection $G \to G/\kN$ (which factors
through $\tits$), and $\chi$ is considered a character on $\kN$ which
is trivial on $H$, cf.\ \cite[Proposition 2.1]{timashev1}.
The function $f=f_{s(\chi)}$ was introduced as a $\kB$-eigenfunction
for $s(\chi)\in\CX_\kB$; we may assume that $f(1)=1$.
According to Lemma \ref{lem-multF}
this implies that $f(bHh')=f(bH)\,f(h') = f(bH)\,\chi(h')^{-1}$.
Since $\kB H$ is dense inside $G$, this means that
$f_{s(\chi)}$ is a $\chi$-eigenfunction for the right $\rightT$-action, too.
Hence, if we multiply the elements of
\[
\Gamma(U,\cL(\chi)) =
\{\zeta \in \bC(G/\kN)\,|\, \divisor \zeta + \tits(\divisor f_{s(\chi)})|_U \geq 0\}
\subset \bC(G/\kN) = \bC(G)^{\kN}
\]
with $f_{s(\chi)} \in \bC(G)^H$, we obtain regular functions on
$\pi^{-1}(U) \subset G$ which are $\kN$-semiinvariant with eigenvalue $\chi$,
namely,
\[
f_{s(\chi)}\cdot \Gamma(U,\cL(\chi)) =
\big\{\eta  \in \cO_G(\pi^{-1}(U))^{H} \kst
\eta \cdot h' = \chi(h')^{-1}\,\eta \big\} = \Gamma(U,\cE(-\chi)).
\vspace{-3ex}
\]
\end{proof}

\subsection{The shift vectors}
\label{sub:rhoD}
Recall that in (\ref{sub:toroidalThm}) we already had met the
section $s:M\to\CX(G/H)$ mentioned
in (\ref{sub:bundleEL}), but there it was used via the dual
cosection
$$
s^*:\CX^*(G/H)\surj N.
$$
In other words,
$(p,s^*):\CX^*(G/H)\stackrel{\sim}{\longrightarrow}
\CX^*(G/\kN)\oplus N$ establishes a splitting of the exact sequence from
Proposition \ref{prop-lattices}.
Note that this proposition also states that
$p(\rho_{\colH})=\rho_{\colN}$ for colors
$\colH\in\Colors(G/H)$ and $\colN=\tits(\colH)\in\Colors(G/\kN)$
where $\rho_{\kbb}$ refers to the elements of
$\CX^*(G/H)$ and $\CX^*(G/\kN)$ induced from the valuations associated to
these colors, respectively.

\begin{definition}
Using the notation from above,
we define for every color $\colN=\tits(\colH)$ its associated
\emph{shift vector}
$$
\shift(\colN)=\shift(\colH):=s^*(\rho_{\colH})\in N.
$$
That is, for $\chi\in M$,
$
\langle \chi, \,\shift(\colN)\rangle =
\langle \chi,\,s^*\rho_{\tits^{-1}(\colN)}\rangle =
\langle s(\chi),\,\rho_{\tits^{-1}(\colN)} \rangle.
$
\end{definition}

The choice of a basis $\cB_M$ of $M$
in (\ref{sub:bundleEL}) allows us to define the
polyhedral cone $\sigma \subset N_\Q:=N \otimes \bQ \cong  \bQ^n$
as the positive orthant in the latter space. This will become the
tail cone for the following important polyhedral divisor.

\begin{proposition}
\label{prop:homSpaceDwnGr}
The vector bundle $E \to G/\kN$ from
{\rm (\ref{sub:bundleEL})} is $\rightT$-equivariantly isomorphic
to $\relTV(\pDiv^{E})$ where
\[
\pDiv^{E} \;=\;
\sum_{\colN \in \Colors(G/\kN)} \hspace{-0.5em}
(\shift(\colN) + \sigma) \otimes \colN\,.
\]
\end{proposition}

\begin{proof}
By Lemma \ref{lem-LE} we can describe the vector bundle $E$ as
\[
\renewcommand{\arraystretch}{1.2}
\begin{array}{rcl}
E \;=\; \Spec_{G/\kN} \sym^\bullet \cE^\vee
&=& \Spec_{G/\kN} \bigoplus_{\chi \in \sigma\dual\cap M} \cE(-\chi)
\\
&=& \Spec_{G/\kN} \bigoplus_{\chi \in \sigma\dual\cap M} \cL(\chi).
\end{array}
\]
However, the very same result is obtained when we analyze the evaluation of the
polyhedral divisor $\pDiv^{E}$ on a multidegree $\chi\in\sigma\dual\cap M$,
namely
$$
\pDiv^{E}(\chi)\;=\;\sum_{\colN}\langle \chi,\,\shift(\colN) \rangle \cdot\colN
\;=\; \sum_{\colN}\langle s(\chi),\,\rho_{\tits^{-1}(\colN)} \rangle \cdot\colN
\;=\; \cL(\chi).
\vspace{-4ex}
$$
\end{proof}

As explained in (\ref{subsec:divFan}),
the $\rightT$-equivariant, open embedding $G/H \subset E$ translates
into a face relation of the corresponding polyhedral divisors. Since
this embedding is induced by $\rightT \subset \bC^m$, it arises from
the face relation $0\faceof\sigma$ among the tail cones.
So, as a corollary, we obtain a description of the polyhedral divisor
$\pDiv^{G/H}$. Note that it depends on the choice of the
section $s$ (hidden in the shift vectors $\shift(\colN)$).
However, in contrast to $E$ and $\pDiv^E$,
it doesn't depend on the choice of a basis of $M$.

\begin{corollary}
\label{cor:homSpaceDwnGr}
The $\rightT$-variety $G/H$ is equal to $\relTV(\pDiv^{G/H})$
where
$$
\pDiv^{G/H} =
\sum_{\colN \in \Colors(G/\kN)} \shift(\colN) \otimes \colN.
$$
In particular, its tail cone is equal to $0$.
\end{corollary}

This completes the proof of Theorem \ref{thm-toroidal}.

\section{The general case}
\label{sec:invCol}

\subsection{}
\label{subsec:prelim1}
Recall that we have fixed $\leftT \subseteq \kB\subseteq G$ such that $\kB H\subseteq G$ is open and dense.
Let $\broots$ denote a basis of the positive roots $\proots$ that correspond to the choice of $\kB$.
In particular, $W=W_G$ is generated by simple reflections $\{s_\alpha\kst \alpha\in\broots\}$.
For every subset $I\subset\broots$ of simple roots let $W_I\subset W$ denote the subgroup
generated by simple  reflections $\{s_\alpha\ |\ \alpha\in I\}$.
The subgroup $W_I$ comes with a distinguished set $W^I\subseteq W$ consisting of the representatives
of minimal length
of the left cosets of $W_I$.
In particular,
$W^I\times W_I\stackrel{\sim}{\to} W$ preserves the minimal
representations as products of simple reflections.
For proofs and further details see \cite{Springer}.

Note that if $G/H$ is of minimal rank then for every simple reflection $s_\alpha\in W$ there exists at most one color
$D=D(\alpha)\in \Colors(G/H)$ such that $s_\alpha(D)\ne D$.
Namely, $D=\ovl{\cO_{s_\alpha u_0}}$ unless $s_\alpha\in W_H$ (see Remark \ref{rem-orbits}).
For an arbitrary subset $\CF \subseteq \Colors(G/H)$ of colors define the subgroup $W_{I(\CF)}\subset W$ by taking
$I(\CF)=\{\alpha\in\broots\ | \ D(\alpha)\in\CF \mbox{ or } s_\alpha\in W_H\}$.
In other words, $W_{I(\CF)}$ is generated by all simple reflections that leave
$\bigcup_{D \in {\Colors(G/H)\setminus\CFC}}D$ invariant.
In particular,
$W_{I(\CF)}=\{w\in W\kst w(\bigcup_{D \in {\Colors\setminus\CFC}}D)=
\bigcup_{D \in {\Colors\setminus\CFC}} D\}$.
Indeed, the subgroup $\{g\in G\ |\ g(\bigcup_{D \in {\Colors\setminus\CFC}}D)=\bigcup_{D \in {\Colors\setminus\CFC}}D\}$ is parabolic, hence, its Weyl group is generated by simple reflections.
For example,
$W_{I(\emptyset)}=W_H$ is the Weyl group of the parabolic subgroup $P=P(\UColorsTor)$
mentioned in (\ref{sub:skeleton}).
The other extremal case
is $W_{I(\Colors(G/H))}=W$, i.e.\ $I(\Colors(G/H))=\broots$.
%


\subsection{}
\label{subsec:simpleWeyl}
Let $X \entspr (C,\CFC)$ be a simple spherical embedding of minimal rank.
Recall that there is an open affine $\rightT$-invariant covering $X = \bigcup_{w \in W} \wh{X}_w$
(see Proposition \ref{prop-covering}).
Note that some of these charts may be identical.
To obtain a non-redundant description we exploit
the subgroup $W_C:=W_{I(\CFC)}$ of $W$ associated to $(C,\CFC)$.
Summing things up, we obtain
$X = \bigcup_{w \in W^{I(\CFC)}} \wh{X}_w$,
where the $\wh{X}_w$ are now pairwise distinct.
\\[1ex]
Let $\pi:\Xtor \longrightarrow X$ (cf.\ diagram from (\ref{sub:comparisonX}))
be the map corresponding to
$(C\cap \valC,\emptyset)\to (C,\CFC)$. Then we have a similar
covering
$\Xtor = \bigcup_{w \in W} \XtorHat_w$.

\begin{lemma}
\label{lem-ClaimB}
For every $w\in W$, the map
$\pi_w:\bigcup_{w'\in w W_C}  \XtorHat_{w'}\to \wh{X}_w$
is the full preimage of $\wh{X}_w$ under
$\pi:\Xtor \longrightarrow X$.
In particular, $\pi_w$ is birational and proper.
\end{lemma}

\begin{proof}
We may assume that $w=\id$ (renaming $w'$ into $w$ afterwards).
It remains to show that
$\pi_{\id}:\bigcup_{w\in W_C} \XtorHat_{w}\to \wh{X}_{\id}$ is a full preimage,
i.e.\ that
$$
\bigcap_{w\in W_C}w\cdot\UColorsTor =
\bigcup_{\colH\in\Colors\setminus\CFC} \pi^{-1}(\ko{\colH}) \eqno(1)
$$
where we had defined
$\UColorsTor:=\bigcup_{\colH\in\Colors} \ko{\colH}\subseteq\Xtor$.
First, let us compare the intersections with the open orbit $G/H$, that is,
show  $$\Omega:=\bigcap_{w\in W_C}w\left(\bigcup_{\colH\in\Colors} \colH\right)= \bigcup_{\colH\in\Colors\setminus\CFC} \colH. \eqno(2)$$
Let $\cO_u$ be a $B$-orbit in $G/H$ that contains a given point $x\in\Omega$ (here we use notation of Section \ref{subsec:minimal_rank}).
Take an increasing path in the graph $\Gamma(G/H)$ from $\cO_u$ to the maximal $B$-orbit $\cO_{u_0}$.
Let $(\alpha_{i_1},\ldots, \alpha_{i_\ell})$ be the labels on its edges,
leading to
$wx\in w(\cO_u)\subset\cO_{u_0}\ne \bigcup_{\colH\in\Colors}\colH$
for $w=s_{\alpha_{i_\ell}}\cdots s_{\alpha_{i_1}}\in W_C$.
Then there exists an $i_j$ such that $i_j\notin I(\CFC)$ since otherwise
we would have $w\in W_C$.
Put $\alpha_{i_j}=\alpha$.
Then Lemma \ref{l.inclusion} implies that
$\cO_u\subset \colH$ for some $\colH\in\Colors\setminus\CFC$, namely, for
$\colH(\alpha)=\ovl{\cO_{s_{\alpha}u_0}}$.
\\[1ex]
Finally, note that the above argument goes through if $G/H$ is replaced by any other $G$-orbit
$\cO\subset \Xtor$ because $\cO$ is also of minimal rank by \cite[Lemma 2.1]{Ressayre}.
More precisely, let $\CFC(\cO)$ be the set of all colors of $\cO$ that are not contained in 
$\bigcup_{\colH\in\Colors\setminus\CFC}\pi^{-1}(\ko{\colH})$.
Identity (2) for $\cO$ and $\CFC(\cO)$ takes form
$$\bigcap_{w\in W_C}w\left(\bigcup_{\colH\in\Colors(\cO)} \colH\right)=
\bigcup_{\colH\in\Colors(\cO)\setminus\CFC(\cO)} \colH. \eqno(3)$$
We have $\cO\cap\UColorsTor=\bigcup_{\colH\in\Colors(\cO)}\colH$ and 
$\cO\cap \bigcup_{\colH\in\Colors\setminus\CFC} \pi^{-1}(\ko{\colH})=\bigcup_{\colH\in\Colors(\cO)\setminus\CFC(\cO)} \colH$.
Hence, both sides of (3) coinside with the respective sides of (1) intersected with $\cO$.
This implies equality (1).
\\[1ex]
Now, birationality and properness of $\pi_w$ follow directly from
the corresponding properties of
$\pi: \Xtor \to X$, cf.\ \cite[Theorem 4.2]{Knop}.
\end{proof}

\subsection{}
\label{subsec:prelim2}
Theorem \ref{thm:main} states that the simple spherical variety $X$
can be described by a divisorial fan $\pFan^X$ whose maximal elements are
indexed by elements of $W$, more precisely
\[
\pFan^X_w = \sum_{a\in\fanY} \hspace{-0.3em}\big(C\cap p^{-1}(a)\big)
\otimes D_a + \sum_{\colN\in\ColorsN}  \hspace{-0.3em}\big(\shift(\colN)
+ (C\cap N_\Q)\big)\otimes \ovl{\colN}
+ \hspace{-0.3em}\sum_{\colN\in\ColorsN\setminus\CF_C}
        \hspace{-0.5em}\emptyset\otimes w\ovl{\colN}.
\vspace{-1ex}
\]
The only difference with respect to $\;\pfan^{X}=\pFan^{\Xtor}$
is that the last sum
runs over $\ColorsN \setminus \CF_C$ instead of the
entire $\ColorsN$. In other
words, we have
\[
\begin{array}{rclll}
\pfan^{X}_w \; = \; \pFan^{\Xtor}_w\hspace{-0.6em}
& = & \cZ & \textnormal{on} &
\Utor_w := Y \setminus \bigcup_{D' \in \ColorsN}
w\overline{D'}\,, \\[.5ex]
\pFan^X_w & = & \cZ & \textnormal{on} & \kU_w := Y \setminus
\bigcup_{D' \in \ColorsN \setminus \CFC} w \overline{D'}\,,
\end{array}
\]
where
$\cZ = \sum_{a\in\fanY} \hspace{-0.3em}\big(C\cap p^{-1}(a)\big)
\otimes D_a + \sum_{\colN\in\ColorsN}  \hspace{-0.3em}\big(\shift(\colN)
+ (C\cap N_\Q)\big)\otimes \ovl{\colN}$
does not depend on $w\in W$.
Since $Y$ is also of minimal rank (cf.\ (\ref{sub:comparison})),
both $\cUtor := \{\Utor_w\kst w\in W\}$ and
$\kUU := \{\kU_w\kst w\in W\}$ are coverings of $Y$.
\begin{lemma}
\label{lem-ClaimA}
The covering $\cUtor$ is a refinement of $\kUU$. In detail,
for every $w\in W$ we have
$\bigcup_{w'\in w W_C} \Utor_{w'} = \kU_w$.
\end{lemma}

\begin{proof}
Similarly to the proof of Lemma \ref{lem-ClaimB},
we may assume that $w=\id$ (renaming $w'$ again into $w$ afterwards).
It remains to show that
$\bigcup_{w\in W_C} \Utor_{w} = \kU_{\id}$,
i.e.\ that
$$
\bigcap_{w\in W_C}w\cdot\UColorsN =
\bigcup_{\colN\in\ColorsN\setminus\CFC} \ko{\colN}
$$
with $\UColorsN:=\bigcup_{\colN\in\ColorsN} \ko{\colN}\subseteq Y$.
However, this claim literally equals the equation we
have shown in the proof of Lemma \ref{lem-ClaimB} for the $X$-level.
Thus, the same arguments apply.
\end{proof}

\subsection{}
\label{subsec:proof}
Our goal now is to
compare the map $\Xtor \longrightarrow X$ (cf.\
diagram from (\ref{sub:comparisonX})) with
the map $\TV(\pfan^X)\to\TV(\pFan^X)$.
The already proven Theorem \ref{thm-toroidal} ensures that the sources of both
maps coincide.
Using $\cZ$, we define
$\cA:= \bigoplus_{u \in (C\cap N_{\bQ})^\vee} \CO(\cZ(u))$
together with the following two affine $\rightT$-varieties
\[
X^{\toroidal}_w := \spec \Gamma(\Utor_w,\cA)\,, \qquad
X_w := \spec \Gamma(\kU_w, \cA)\,.
\]
They are open subsets of $\TV(\pfan^X)$ and $\TV(\pFan^X)$, respectively.
Everything fits now into the following commutative diagram:
$$
\xymatrix@R=3ex{
\Xtor \ar@{=}[d] & \ar@{_(->}[l]
\hspace{0.3em}\bigcup_{w'\in w W_C}  \XtorHat_{w'}
\ar@{=}[d] \ar[r]^-{\pi_w} &
\wh{X}_w  \ar@{^(->}[r] & X\\
\TV(\pfan^{X}) & \ar@{_(->}[l]
\hspace{0.3em}\bigcup_{w'\in w W_C} \Xtor_{w'}
\ar[r]^-{\psi_w} & X_w \ar@{^(->}[r] & \TV(\pFan^{X})
}
$$
While the vertical equalities
$\XtorHat_{w'} = \Xtor_{w'}$ come from Theorem \ref{thm-toroidal},
we have seen in the
Lemmata \ref{lem-ClaimB} and \ref{lem-ClaimA} that both central,
horizontal maps $\pi_w$ and $\psi_w$ are birational and proper.
On the other hand, both $\wh{X}_w$ and $X_w$ are affine --
hence, they have to coincide, too.
This proves Theorem~\ref{thm:main}.

\subsection{A counterexample for the non-minimal rank case}
\label{subsec:minRankCountEx}
Note that Proposition \ref{prop-covering} was essential for the proof.
The assumption that $G/H$ is of minimal rank in Theorem \ref{thm:main} was used both in Proposition \ref{prop-covering} and in Lemma \ref{lem-ClaimB}, and can not be dropped as can be seen from the following example with a non-trivial $\rightT$-action.

\begin{example}\label{ex.GL_2/H}
Take $G=GL_2$ and consider its action
on $X=\bP^1\times \bP^2$ by matrices
$$
\left(\begin{array}{cc} a & b \\ c & d \end{array}\right)\times
\left(\begin{array}{ccc} a & b & 0 \\ c & d & 0 \\ 0 & 0 & 1\\ \end{array}\right).
$$
It is easy to check that $X$ together with the point $(1:0)\times(0:1:1)$ is a spherical embedding of $G/H$, where
$$
H=\left\{\left(\begin{array}{cc} \lambda & 0 \\ 0 & 1 \end{array}\right)
\kST \lambda\in\C^*\right\}.
$$
Note that $G/H$ is not of minimal rank.
The $\rightT$-variety given by Theorem \ref{thm:main} in this case does not coincide with the whole $X$ since the former is not complete (note that the base space $Y$ in this case is exactly
the $G$-variety $\bP^1\times\bP^1$ from Example \ref{ex.P^1_P^1}).
Namely, Theorem \ref{thm:main} yields only four out of six standard affine charts for
$\bP^1\times \bP^2$.
\end{example}

\section{Examples}
\label{sec:examples}

\subsection{Horospherical varieties}
\label{subsec:exhorospherical}
We use the notation introduced in (\ref{subsec:horospherical}).
Also, recall from loc.cit.\ that $\valC = \CX(G/H)^*_\Q$ for any horospherical
embedding $G/H \subset X$.
Hence, our polyhedral divisors $\pFan^X$ will be defined on the flag variety $G/P$.
Note also that for horospherical varieties the exact sequence
$$
0 \to  N  \to \CX^*(G/H) \stackrel{p}{\rightarrow} \CX^*(G/P)\to 0
$$
reduces to the canonical isomorphism $N\simeq \CX^*(G/H)$, that is, there is no need to choose a splitting.
The following examples are taken from \cite{pasquier1}.

\subsubsection{Embeddings of $\Sl_2/U$}
This example continues Example \ref{ex.SL_2/U} and relates the colored fans on Figure 1 to their corresponding divisorial fans.
The Weyl group of $SL_2$ contains only two elements, $\id$ and $w$.
Using Theorems \ref{thm-toroidal}, \ref{thm:main}, and identifying
the color $D' = \{0\}$ and $wD'= \{\infty\}$ on $\bP^1$, we obtain the
following maximal elements of the respective divisorial fans:
\[
\begin{array}{rcl}
\pFan^{(a)}([0,\infty),\id) & = & \emptyset \otimes 0 \\
\pFan^{(a)}\big([0,\infty),w\big) & = & [1,\infty) \otimes 0
+ \emptyset \otimes \infty \\[1.5ex]

\pFan^{(b)}\big(([0,\infty),\alpha),\{\id,w\}\big) & = &
[1,\infty) \otimes 0 \\[1.5ex]

\pFan^{(c)}\big((-\infty,0],\id \big) & = & \emptyset \otimes 0 \\
\pFan^{(c)}\big((-\infty,0],w \big) & = & (-\infty,1] \otimes 0
+ \emptyset \otimes \infty \\[1.5ex]

\pFan^{(d)}\big([0,\infty),\id \big) & = & \emptyset \otimes 0 \\
\pFan^{(d)}\big([0,\infty),w \big) & = & [1,\infty) \otimes 0
+ \emptyset \otimes \infty \\
\pFan^{(d)}\big((-\infty,0],\id \big) & = & \emptyset \otimes 0 \\
\pFan^{(d)}\big((-\infty,0],w \big) & = & (-\infty,1] \otimes 0
 + \emptyset \otimes \infty \\[1.5ex]

\pFan^{(e)}\big(([0,\infty),\alpha),\{\id,w\}\big) & = &
[1,\infty) \otimes 0 \\
\pFan^{(e)}\big((-\infty,0],\id \big) & = & \emptyset \otimes 0 \\
\pFan^{(e)}\big((-\infty,0],w \big) & = & (-\infty,1] \otimes 0
+ \emptyset \otimes \infty

\end{array}
\]

They are all toric, and it can easily be verified that the torus action is the
action of a subtorus given by the following exact sequence of lattices of
one-parameter subgroups:
\[
\xymatrix{0 \ar[r] & \bZ \ar[r]^{\phi} & \bZ^2 \ar[r]^{\pi} \ar@/^/[l]^{\sigma} &
\bZ \ar[r] & 0 }
\]
with
\[
\begin{array}{ccc} \phi = \left(\begin{array}{c} 1 \\ 1 \end{array} \right), &
\pi = \left(\begin{array}{cc} 1 & -1 \end{array} \right), & \sigma =
\left(\begin{array}{cc} 1 & 0 \end{array} \right). \end{array}
\]

So we are in fact in a toric downgrade situation as described in
(\ref{subsec:torDown}). The Chow quotient $Y$ is $\bP^1$, and one can
check that applying the recipe from (\ref{subsec:torDown})
yields the same divisorial fans as described above.

\subsubsection{An example of rank 1 and $\rightT$-complexity 3}
Let $\kBalt\subset SL_3$ denote the subgroup of lower-triangular matrices.
We consider the subgroup $H \subset \kBalt$ of matrices whose second diagonal entry is 1.
This yields a four-dimensional horospherical homogeneous space $G/H$ of rank one over the full flag variety $G/\kBalt$.
There are four complete embeddings, but we will only have a closer look
at two of them, namely those whose colored fans are given in
Figure \ref{fig:CFembSL3/H}.

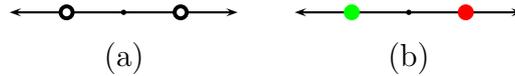
\begin{figure}[h]
\psset{unit=0.75cm}
\begin{pspicture}(0,-2)(10,0)

\psline{<->}(0,-1)(4,-1)
\qdisk(2,-1){1pt}
\qdisk(3,-1){3pt}
\psset{linecolor=white}
\qdisk(3,-1){1.5pt}
\psset{linecolor=black}
\qdisk(1,-1){3pt}
\psset{linecolor=white}
\qdisk(1,-1){1.5pt}
\psset{linecolor=black}
\uput[270](2,-1.3){(a)}

\psline{<->}(5,-1)(9,-1)
\qdisk(7,-1){1pt}
\psset{linecolor=red}
\qdisk(8,-1){3pt}
\psset{linecolor=green}
\qdisk(6,-1){3pt}
\psset{linecolor=black}
\uput[270](7,-1.3){(b)}

\end{pspicture}
\caption{Colored fans associated to complete embeddings of $\SL_3/H$.}
\label{fig:CFembSL3/H}
\end{figure}

Let $\alpha$, $\beta$ denote simple roots of $SL_3$.
The Weyl group is isomorphic to $S_3$. It is generated by the reflections $s_\alpha$ and  $s_\beta$ and consists of 6 elements: $1$, $s_\alpha$, $s_\beta$, $s_\alpha s_\beta$,
$s_\beta s_\alpha$, $s_\alpha s_\beta s_\alpha$.
Let $D'_\alpha$ and $D'_\beta$ denote the colors of $G/B$.
The $W$-action maps $D'_\alpha$ to $s_\alpha D'_\alpha$ or $s_\beta s_\alpha D'_\alpha$ (note that $s_\beta D'_\alpha=D'_\alpha$), and $D'_\beta$ to $s_\beta D_\beta$ or $s_\alpha s_\beta D_\beta$ (again,  $s_\alpha D'_\beta=D'_\beta$).
The following tables in Figure \ref{fig:divFansCFembSL3/H} encode the maximal elements of the corresponding divisorial fans.

\begin{figure}[h]
\[
\begin{array}{c||c|c|c}
\mathbf{(a)} & v & (-\infty,0] & [0,\infty) \\
\hline \hline
{D'}_\alpha & 1 & \alpha,\alpha\beta,\beta\alpha,\alpha\beta\alpha&
\alpha,\alpha\beta,\beta\alpha,\alpha\beta\alpha\\
s_\alpha{D'}_\alpha& 0 & 1, \beta, \beta\alpha, \alpha\beta\alpha &
1, \beta, \beta\alpha, \alpha\beta\alpha \\
s_\beta s_\alpha D'_\alpha& 0 & 1,\alpha,\beta,\alpha\beta&
1,\alpha,\beta,\alpha\beta \\
{D'}_{\beta} & -1 & \beta,\alpha\beta, \beta\alpha,\alpha\beta\alpha &
\beta,\alpha\beta, \beta\alpha,\alpha\beta\alpha\\
s_\beta{D'}_{\beta} & 0 &  1,\alpha, \alpha\beta,\alpha\beta\alpha &
1,\alpha, \alpha\beta,\alpha\beta\alpha\\
s_\alpha s_\beta{D'}_{\beta} & 0 & 1,\alpha,\beta,\beta\alpha &
1,\alpha,\beta,\beta\alpha
\end{array}
\]

\[
\begin{array}{c||c|c|c}
\mathbf{(b)} & v & (-\infty,0] & [0,\infty) \\
\hline \hline
{D'}_\alpha & 1 &  \alpha, \alpha\beta,\beta\alpha,\alpha\beta\alpha&
{\red 1},\alpha, {\red \beta}, \alpha\beta,\beta\alpha,\alpha\beta\alpha\\
s_\alpha{D'}_\alpha& 0 & 1, \beta, \beta\alpha, \alpha\beta\alpha &
1, {\red \alpha}, \beta, {\red \alpha\beta}, \beta\alpha, \alpha\beta\alpha \\
s_\beta s_\alpha D'_\alpha& 0 & 1,\alpha,\beta,\alpha\beta&
1,\alpha,\beta,\alpha\beta, {\red \beta\alpha}, {\red \alpha\beta\alpha} \\
{D'}_{\beta} & -1 & {\green 1}, {\green \alpha}, \beta,\alpha\beta, \beta\alpha,\alpha\beta\alpha & \beta,\alpha\beta, \beta\alpha,\alpha\beta\alpha\\
s_\beta{D'}_{\beta} & 0 &  1,\alpha, {\green \beta}, \alpha\beta, {\green \beta\alpha}, \alpha\beta\alpha &
1,\alpha, \alpha\beta,\alpha\beta\alpha\\
s_\alpha s_\beta{D'}_{\beta} & 0 & 1,\alpha,\beta, {\green \alpha\beta}, \beta\alpha, {\green \alpha\beta\alpha} &
1,\alpha,\beta,\beta\alpha
\end{array}
\]
\caption{Divisorial fans associated to complete embeddings of $\SL_3/H$.}
\label{fig:divFansCFembSL3/H}
\end{figure}

They are to be read as follows: each row is indexed by a divisor
$\overline{D'}_\bullet$. The corresponding one-dimensional slice
is subdivided at $v$ into two unbounded components. The labels of these
components are given in columns 3 and 4, respectively. Note that
we use a shorthand notation for the labels, i.e.\ $\bullet = s_\bullet$.

\subsection{$(\Gl_2\times\Gl_2)$-equivariant embeddings of $\Gl_2$}
\label{subsec:ExGlZw}
These examples are classical yet we choose to discuss all details in order to give
the reader the possibility to recall all notions which have
been defined so far.

\subsubsection{Basic setup}
\label{subsub:ExGBasic}
Let $G:=\Gl_2\times\Gl_2$ act on $\Gl_2$ by left and right multiplications, that is $(g_1,g_2):g\mapsto g_1gg_2^{-1}$.
It follows that $H:=\Delta(\Gl_2)$ where $\Delta$ denotes the diagonal embedding of $\GL_2$ to $G$.
We fix the Borel subgroup $B:=B_{\GL_2}^+\times B_{\GL_2}^- \subset G$, where $B_{\GL_2}^+$ and $B_{\GL_2}^-$
consist of upper and lower triangular matrices, respectively.
Furthermore we fix the maximal torus $\leftT\subseteq B$ given by
the diagonal matrices $(\C^*)^2\times(\C^*)^2$. Hence, we have that
$\CX_B=\CX_{\leftT}=\Z^4$ with basis $\{e^+_1,e^+_2,e^-_1,e^-_2\}$.
Finally,  $U:=B_u$ denotes the unipotent radical of $B$. As usual,
elements of $\Gl_2$ are denoted by matrices
$$
\left(\begin{array}{cc} a & b \\ c & d \end{array}\right).
$$
One easily checks that $\C(G/H)=\C(\Gl_2)=\C(a,b,c,d)$ and
$\C(\Gl_2)^U=\C(d,\det)$ with $\det=ad-bc$.
Using the exact sequence from (\ref{sub:ColFans}), we see that
the weights of these generators are
$\chi(d)= e^-_2-e^+_2$ and
$\chi(\det)= (e^-_1-e^+_1)+( e^-_2-e^+_2)$.
%
The Weyl group of $G$ is $W=\{1,s_\alpha,s_\beta,s_\alpha s_\beta\}$ with
$$
s_{\alpha} =
\left(\begin{array}{cccc} 0 & 1 & 0 & 0 \\ 1 & 0 & 0 & 0 \\
0 & 0 & 1 & 0 \\ 0 & 0 & 0 & 1 \end{array}\right)
\quad \textnormal{and} \quad
s_{\beta} =
\left(\begin{array}{cccc} 1 & 0 & 0 & 0 \\ 0 & 1 & 0 & 0 \\
0 & 0 & 0 & 1 \\ 0 & 0 & 1 & 0 \end{array}\right),
$$
i.e.,\ $s_\alpha:\, e^+_1 \leftrightarrow e^+_2$ and
$s_\beta:\, e^-_1 \leftrightarrow e^-_2$.

\subsubsection{Further ingredients}
The Bruhat decomposition of $\Gl_2$ with $W_{\Gl_2}=\{1,s\}$ yields
$\Gl_2=(B^+ B^-) \disjcup (B^+sB^-)$.
The first double class is the open orbit and
shows that $G/H=\Gl_2$ is spherical whereas the second
double class
corresponds to the unique color $\colH=V(d)$ in $\Colors$.

The normalizer $\kN$ is equal to $\Delta(\Gl_2)\cdot (\C^*\times\C^*)$.
The transitive $G$-action on $\Gl_2$ from (\ref{subsub:ExGBasic}) induces a
transitive $G$-action on $\PGL_2$ providing the isomorphism
$G/\kN=\PGL_2$. Its unique and wonderful compactification
is $\PP^3=\PGL_2\disjcup V(\det)$.

\subsubsection{The ambient spaces for the (colored) fans}
\label{subsub:ExGAmbient}
We identify $\rightT = \kN/H$ with $\C^*$ via $(t,1)\mapsto t$.
In particular, $M=\Z$, and we derive from the commutative diagram of
Lemma \ref{lem-multF} that the sequence
$$
\;0 \to \CX(G/\kN) \to \CX(G/H) \to M \to 0
$$
sends $(e^+_i-e^-_i)\mapsto 1$ for $i=1,2$,
cf.\ proof of Proposition \ref{prop-lattices}. The kernel
$\CX(G/\kN)$ is generated by $(e^-_1-e^+_1)-( e^-_2-e^+_2)$
which is $\chi(\det/d^2)$. The dual sequence
$$
0 \to  (N=\Z) \to \CX^*(G/H) \stackrel{p}{\longrightarrow} \CX^*(G/\kN)\to 0
$$
sends $1\mapsto -E^1-E^2$ and $E^1\mapsto E$, $E^2\mapsto -E$
where $\{E^1,E^2\}$ and $\{E\}$ are the dual bases of
$\{(e^-_1-e^+_1),\, ( e^-_2-e^+_2)\}=\{\chi(\det\!/d),\, \chi(d)\}$
and $\{(e^-_1-e^+_1)-( e^-_2-e^+_2)\}=\{\chi(\det\!/d^2)\}$, respectively.
Since the valuation $v_{\colH}=v_{V(d)}$ sends
$\det\!/d\mapsto -1$ and $d\mapsto 1$, we obtain $\rho_{\colH}=E^2-E^1$.
We fix the splitting $E\mapsto E^1$. This induces the projection
$\CX^*(G/H)\surj N=\Z$ with $E^1\mapsto 0$ and $E^2\mapsto -1$.
In particular,
the shift vector from (\ref{sub:toroidalThm}) equals $\shift(\colH)=-1$.

\subsubsection{The valuation cone}
\label{subsub:ExGVal}
Let us consider the $\Gl_2$-embeddings given in the following diagram
$$
\xymatrix@!0@C=5em@R=7.5ex{
& \til{\C^4} \ar[d]^-{\pi}\ar@{^(->}[r] & \til{\PP^4} \ar[d]^-{\pi}\\
\Gl_2 \ar@{^(->}[r] \ar@{^(->}[ur] & \C^4 \ar@{^(->}[r] & \PP^4
}
$$
where the upper row consists of
the blow ups at the origin of the
corresponding varieties in the lower row.
This picture provides us with three $G$-invariant divisors and their
associated valuations:
$v_{\det}\entspr V(\det)=\C^4\setminus\Gl_2$,\,
$v_{E}\entspr E=\pi^{-1}(0)$, and
$v_{\infty} \entspr \PP^4\setminus\C^4$. They send the equations
$\det\!/d=(x_1x_4-x_2x_3)/(x_0x_4)$ and $d=x_4/x_0$ to
$(1,0)$, $(1,1)$, and $(-1,-1)$, respectively.
This means that
$\rho_{\det}=E^1$,
$\rho_{E}=E^1+E^2$, and
$\rho_{\infty}=-(E^1+E^2)$.
These elements span the valuation cone
$$
\valC = \{w_1E^1+w_2E^2\kst w_1\geq w_2\} \subseteq \CX^*(G/H),
$$
i.e.\ the lower half plane which is bounded by the line
$\langle E^1+E^2\rangle$.

\begin{figure}[h]
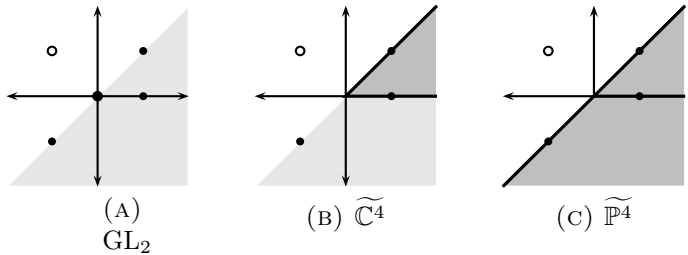

\centering
\subfloat[$\Gl_2$]{\valuationConeGlTwo}\hspace*{4ex}
\subfloat[$\til{\C^4}$]{\BUaffineFour}\hspace*{4ex}
\subfloat[$\til{\bP^4}$]{\BUprojectiveFour}
\caption{Toroidal $\Gl_2 \times \Gl_2$-equivariant embeddings of $\Gl_2$}
\label{fig:BUcompGlTwo}
\end{figure}

\subsubsection{The colored fans}
\label{subsub:ExColF}
All upper embeddings $\Gl_2\subseteq\til{\C^4}\subseteq\til{\PP^4}$
are toroidal. The uncolored cone of $\Gl_2$ is equal to $\{0\}$, the
one corresponding to $\til{\C^4}$ equals $\langle E^1,\,E^1+E^2\rangle$,
whereas $\til{\bP^4}$ is given by the complete subdivision of $\valC$
by the ray $\langle E^1\rangle$. Hence, it consists of the two
uncolored cones $\langle -(E^1+E^2),\,E^1\rangle$ and
$\langle E^1,\, E^1+E^2\rangle$, see Figure \ref{fig:BUcompGlTwo}.

Blowing down the exceptional divisor $E$ via $\pi$ gives us two
non-toroidal spherical embeddings of $\Gl_2$, namely $\C^4$ and $\bP^4$.
All we have to do is to replace the uncolored cone
$\langle E^1,\,E^1+E^2\rangle$ appearing in both blow ups
by the colored one $\big(\langle E^1,\,E^2-E_1\rangle,\; \{\colH\}\big)$,
cf.\ Figure \ref{fig:compGlTwo}.

\begin{figure}[h]
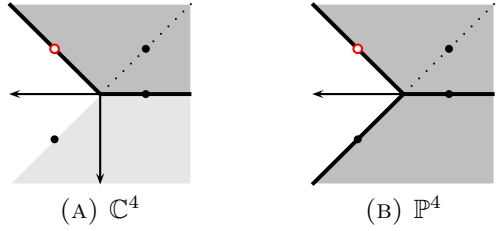

\centering
\subfloat[$\C^4$]{\affineFour}\hspace*{8ex}
\subfloat[$\bP^4$]{\projectiveFour}
\caption{Non-toroidal $\Gl_2 \times \Gl_2$-equivariant
embeddings of $\Gl_2$}
\label{fig:compGlTwo}
\end{figure}

\subsubsection{The divisorial fans}
\label{subsub:ExDivF}
The induced action of $\rightT=\C^*$ on $\C^4$ corresponds, up to sign,
to the standard $\Z$-grading of the affine coordinate ring
$\C[a,b,c,d]$.
Performing the usual downgrading procedure for the diagonal subtorus
$\C^*\hookrightarrow (\C^*)^4$, we see that the polyhedral divisor
$\pDiv$ for $\C^4$ is defined over $Y=\PP^3$ and equal to
$[1,\infty)\otimes H$ where $H=H_0\subseteq \PP^3$ denotes a hyperplane.
Actually, the toric downgrade yields
$$
\pDiv=[1,\infty)\otimes H_0 + \sum_{i=1}^3 [0,\infty)\otimes H_i
$$
with $H_i=V(z_i)$. However, since one can omit trivial summands,
i.e.\ those having just the tail cone as their coefficient, we arrive
at the description from above.
The coefficients of $H_i$ in the extended version
arise as intersections of the four affine lines
$\ell_i=e^i+\Q\cdot e$ (with $e:=\sum_i e^i$)
with the upper orthant $\Q^4_{\geq 0}=\langle e^0,\ldots,e^3\rangle$
that represents $\C^4$ as an affine toric variety.

Blowing up $0\in\C^4$, we obtain $\til{\C^4}=\ttoric(\pDiv)=\toric(\pFan)$
where the four maximal elements of the divisorial fan
$\pFan=\{\pDiv_0,\ldots,\pDiv_3\}$ are given by
$\pDiv_i:=\pDiv + \emptyset\otimes H_i$. On the one hand, this simulates
the relative $\Spec$ construction via the affine open covering
$\{\PP^3\setminus H_i\}$. On the other hand, it arises naturally from the
toric downgrade construction. Namely, for $\C^4$,
we had intersected the four lines $\ell_i$
with a single polyhedral cone.
But for $\til{\C^4}$, we subdivide
$\Q^4_{\geq 0}$ into four chambers $C_i$ by
inserting the new ray $e=\sum_i e^i$.
These smaller cones correspond to the $\pDiv_i$.
Since each of the four lines $\ell_i$ misses exactly one of them,
namely $\ell_i\cap C_i=\emptyset$,
we obtain $\emptyset$ as the coefficient of $H_i$ in $\pDiv_i$.

Representing $\PP^4$ and $\til{\PP^4}$ as toric varieties involves four
additional cones,
respectively. Hence, our toric downgrade creates another four p-divisors
$\{\pDiv'_0,\ldots,\pDiv'_3\}$ following the same
pattern as for the case of $\til{\C^4}$. Their common tail cone
becomes $(-\infty,0]$:
$$
\pDiv_i'=(-\infty,1] \otimes H_0 + \emptyset\otimes H_i\,.
$$

\subsubsection{The Grassmannian $\Grass(2,4)$}
\label{subsub:ExGGrass}
Let $W_i:=\C^2$ ($i=1,2$) be two copies of the very same
complex plane $\C^2$. By $G\subseteq\Gl_4$ we see that $G$ acts on
$\C^4=W_1\oplus W_2$, and therefore also on $\Grass(2,4)$. Since $G$
respects the decomposition of $\C^4$, its orbits are given by
$\orb(d_1,d_2):=\{V\in \Grass(2,4)\kst \dim (V\cap W_i)=d_i\}$. The
following list displays all pairs $(d_1,d_2)$ which give a non-empty
orbit: \\[-1ex]

\begin{center}
\begin{tabular}{c|c|c|c|c}
dim & 0 & 2 & 3 & 4 \\
\hline
$(d_1,d_2)$ & (2,0),(0,2) & (1,1) & (1,0),(0,1) & (0,0)\\[.5ex]
\end{tabular}
\end{center}

Let $V_0:=\{(v,v)\kst v\in\C^2=W_i\}\in\orb(0,0)$.
Then $\stab_{V_0}=\Delta(\Gl_2)$,
i.e.\ $1_G\mapsto V_0$ provides an embedding
$\Gl_2\hookrightarrow \Grass(2,4)$ of the usual type.
Its colored fan is presented in Figure \ref{fig:GrassTwoFour}.

\begin{figure}[h]
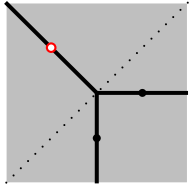

\centering
\GrassTwoFour
\caption{Colored fan of $\Grass(2,4)$}
\label{fig:GrassTwoFour}
\end{figure}

Using Pl\"ucker coordinates, the induced $\rightT$-action on $\Grass(2,4)$
can be obtained from $\deg x_{01}=1$, $\deg x_{23}=-1$, and
$\deg=0$ for the remaining variables. Then the resulting p-divisor for
$\C^6$ as well as the divisorial fan $\pFan$ for $\PP^5$
live on the four-dimensional weighted projective space $\PP(2,1,1,1,1)$.
Denoting its hyperplanes by $A,H_1,\ldots,H_4$, the slices are given by
$$
\pFan_A=(-\infty,0,1,\infty),\,
\pFan_{H_i}=(-\infty,0,\infty)\; (i=1,2,3), \mbox{ and }
\pFan_{H_4}=(-\infty,-1,\infty).
$$
The divisorial fan $\pFan$ is generated
by six p-divisors $\pDiv^0,\ldots,\pDiv^5$ corresponding to the standard
affine, open
covering of $\PP^5$. They can be visualized as labels on the cells of the
slices:
All cells $(-\infty,\kbb]$ carry the label $\pDiv^0$ and, similarly,
all $[\kbb,\infty)$ belong to $\pDiv^5$. All middle cells,
namely $[0,1]$ in $\pFan_A$ and the vertices in
$\pFan_{H_i}$, are labeled with $\{\pDiv^1,\ldots,\pDiv^4\}\setminus \pDiv^i$,
i.e.\ the $H_i$-coefficient in $\pDiv^i$ is $\emptyset$.
\\[1ex]
Finally, the embedding $\Grass(2,4)\hookrightarrow \PP^5$
corresponds to the embedding $\PP^3\hookrightarrow \PP(2,1,1,1,1)$,\,
$(a:b:c:d)\mapsto \big( (bc-ad):a:b:c:d)$ on the level of Chow quotients.
Hence, the divisorial fan of $\Grass(2,4)$ is the restriction of $\pFan$
to $\PP^3$. In particular, $H_1,\ldots,H_4$ become the standard
hyperplanes, and $A$ turns into the quadric $\det\subseteq\PP^3$.

\subsubsection{Comparison of divisorial and colored fans}
\label{subsub:ExGComp}
The slices of the divisorial fan of $\til{\PP^4}$ on $\PP^3$
from (\ref{subsub:ExDivF}) are either
$(-\infty,1,\infty)$ or $(-\infty,0,\infty)$ with four separate labels
for both the negative and positive side. This labeling together with the
presence of empty coefficients corresponds exactly to
the divisorial fan introduced in Definition \ref{def-pFanX}
in (\ref{sub:toroidalThm}):
The first summand involves the only $G$-invariant divisor
$V(\det)\subset\PP^3$. Since its coefficient equals a shift of the tail fan
$(-\infty,0,\infty)$, this sum can be incorporated in the other summands,
involving the only color $V(d)\subseteq\PP^3$.
Indeed, since the Weyl group has four elements,
both top-dimensional cells appear exactly four times.
\\[1ex]
Comparing this with the divisorial fan of $\PP^4$ on $\PP^3$,
we see in (\ref{subsub:ExDivF}) that the four different labels on the one side
merge into one common label. This reflects exactly the description
of the divisorial fan from Definition \ref{def-pFanX}:
Since $\ColorsN\setminus\CF=\emptyset$,
the last sum becomes void for these cells.
\\[1ex]
Finally, we consider the colored fan of
$\Grass(2,4)$, see Figure \ref{fig:GrassTwoFour}. It is induced from
the subdivision of $\valC$ by two rays, namely those spanned
by $E^1$ and $-E^2$, respectively. Note that there are two maximal cones
which are not contained in $\valC$ since they contain the color as a
generator. This means that $\ColorsN\setminus\CF=\emptyset$
occurs now on both sides -- creating the simple labelings by $\pDiv^0$ and
$\pDiv^5$ in (\ref{subsub:ExGGrass}). Moreover, the polyhedral coefficient
$\pFan_A$ clearly is the intersection of the colored fan with an affine
line within the valuation cone. However, this summand cannot be
incorporated in the others as it was possible for $\PP^4$. The reason is
that it carries a richer structure as just being a shift of the tail fan.



\begin{thebibliography}{KKMSD73}


\bibitem[AB04]{toricDegSpherical}
V. Alexeev and M. Brion.
\newblock Toric degenerations of spherical varieties.
\newblock {\em Sel. Math., New Ser.}, 10(4):453--478, 2004.

\bibitem[AH03]{tvar0}
K. Altmann and J. Hausen.
\newblock Polyhedral divisors and algebraic torus actions, extended version.
\newblock arXiv:math/0306285v1, 2003.

\bibitem[AH06]{tvar1}
K. Altmann and J. Hausen.
\newblock {Polyhedral Divisors and Algebraic Torus Actions}.
\newblock {\em Math. Ann.}, 334:557--607, 2006.

\bibitem[AHS08]{tvar2}
K. Altmann, J. Hausen, and H. S{\"u}{\ss}.
\newblock {Gluing Affine Torus Actions Via Divisorial Fans}.
\newblock {\em Transformation Groups}, 13(2):215--242, 2008.

\bibitem[Bra10]{Bravi}
P.~Bravi.
\newblock {Classification of spherical varieties.}
\newblock volume~1 of {\em Les
  cours du CIRM, 1 no. 1: Actions hamiltoniennes : invariants et
  classification}, 2010.

\bibitem[Bri97]{BrionFrench}
M. Brion.
\newblock {Vari\'{e}t\'{e}s sph\'{e}riques.}
\newblock Notes de la session de la S. M. F. "Op\'{e}rations hamiltoniennes et
  op\'{e}rations de groupes alg\'{e}briques" (Grenoble),\\
\newblock available at http://www-fourier.ujf-grenoble.fr/\verb=~=mbrion/notes.html,
  1997.

\bibitem[Bri01]{BrionOrbits}
M. Brion
\newblock On orbit closures of spherical subgroups in flag varieties.
\newblock {\em Comment. Math. Helv.}, 76(2):263-299, (2001)


\bibitem[Bri07a]{BrionLogHom}
M. Brion.
\newblock Log homogeneous varieties.
\newblock In Walter (ed.) et~al. Ferrer~Santos, editor, {\em Actas del XVI
  coloquio Latinoamericano de \'algebra}, pages 1--39. Revista Matem\'atica
  Iberoamericana, 2007.

\bibitem[Bri07b]{coxWonderful}
M. Brion.
\newblock The total coordinate ring of a wonderful variety.
\newblock {\em J. Algebra}, 313(1):61--99, 2007.

\bibitem[Bri09]{Bremen}
M. Brion.
\newblock Spherical varieties.
\newblock Notes of a mini course by R.~Devyatov, D.~Fratila, V.~Tsanov,
\newblock  {\em Highlights in Lie algebraic methods}, 3--24, Progr. Math., 295, Birkh\"auser/Springer, New York, 2012

\bibitem[DCP83]{cp1}
C.~De~Concini and C.~Procesi.
\newblock Complete symmetric varieties.
\newblock volume 996 of {\em Lecture Notes
  in Mathematics}, Springer, Berlin, 1983.

\bibitem[DCP85]{cp2}
C.~De~Concini and C.~Procesi.
\newblock { Complete symmetric varieties II ({I}ntersection {T}heory).}
  \newblock volume~6 of {\em Algebraic groups and related topics}, 1985.

\bibitem[Gag]{coxSpherical}
G.~Gagliardi.
\newblock The {C}ox ring of a spherical embedding.
\newblock {\em J. Algebra}, 397(1):548--569, 2013.

\bibitem[HS10]{tvarcox}
J. Hausen and H. S{\"u}{\ss}.
\newblock The {C}ox ring of an algebraic variety with torus action.
\newblock {\em Adv. Math.}, 225(2):977--1012, 2010.

\bibitem[Kno91]{Knop}
F.~Knop.
\newblock {The Luna-Vust Theory of Spherical Embeddings}.
\newblock {\em Proceedings of the Hyderabad Conference on Algebraic Groups,
  Manoj-Prakashan}, pages 225--249, 1991.


\bibitem[LV83]{LV}
D.~Luna and Th. Vust.
\newblock Plongements d'espaces homog\`enes.
\newblock {\em Comment. Math. Helv.}, 6:186--245, 1983.

\bibitem[Pas]{pasquier1}
B. Pasquier.
\newblock {Vari\'{e}t\'{e}s horosph\'{e}riques de Fano}.
\newblock {\em Bull. Soc. Math. France} 136(2):195--225, 2008

\bibitem[Res]{Ressayre}
N. Ressayre.
\newblock {Spherical homogeneous spaces of minimal rank}.
\newblock {\em Adv. Math.}, 224(5):1784--1800, 2010.


\bibitem[Spr]{Springer}
T.A. Springer.
\newblock Linear Algebraic Groups.
\newblock Number~9 in Progress in Mathematics. Birkh\"auser-Verlag, Boston.


\bibitem[Tim11]{timashev1}
D.A. Timashev.
\newblock Homogeneous spaces and equivariant embeddings,
\newblock volume 138 of
  {\em Encyclopaedia of Mathematical Sciences}, Springer, Berlin, 2011.

\end{thebibliography}
\end{document}